\documentclass[12pt]{amsart}

\usepackage{amsmath,amssymb,latexsym,esint,cite,mathrsfs}
\usepackage{verbatim,wasysym}
\usepackage[left=2.6cm,right=2.6cm,top=2.8cm,bottom=2.8cm]{geometry}
\usepackage{tikz,enumitem,graphicx, subfig, microtype, color}
\usepackage{epic,eepic}

\allowdisplaybreaks
\usepackage[colorlinks=true,urlcolor=blue, citecolor=red,linkcolor=blue,
linktocpage,pdfpagelabels, bookmarksnumbered,bookmarksopen]{hyperref}
\usepackage[hyperpageref]{backref}
\usepackage[english]{babel}

\numberwithin{equation}{section}

\newtheorem{thm}{Theorem}[section]
\newtheorem{lem}[thm]{Lemma}
\newtheorem{cor}[thm]{Corollary}
\newtheorem{pro}[thm]{Proposition}

\newtheorem{rem}[thm]{Remark}

\newcommand{\N}{\mathbb{N}}

\newcommand{\R}{\mathbb{R}}

\title[Conical diffusion]{Liouville results for semilinear integral equations with conical diffusion}

\author{Isabeau Birindelli, Lele Du and Giulio Galise}
\address{Dipartimento di Matematica Guido Castelnuovo, Sapienza
Universit\`a di Roma, Piazzale Aldo Moro 5, 00185, Roma, Italy.}
\email[I. Birindelli]{isabeau@mat.uniroma1.it}
\email[L. Du]{lele.du@uniroma1.it}
\email[G. Galise]{galise@mat.uniroma1.it}

\begin{document}

\def\ON{\text{ on }}

\begin{abstract}
Nonexistence results for positive supersolutions of the equation $$-Lu=u^p\quad\text{in $\R^N_+$}$$  are obtained,    $-L$ being any symmetric and  stable linear operator, positively homogeneous of degree $2s$, $s\in(0,1)$, whose spectral measure is absolutely continuous and positive only in a relative open set of the unit sphere of $\R^N$. The results are sharp: $u\equiv 0$ is the only nonnegative supersolution in the subcritical regime $1\leq p\leq\frac{N+s}{N-s}\,$, while  nontrivial supersolutions exist, at least for some specific $-L$, as soon as $p>\frac{N+s}{N-s}$. \\ The arguments used rely on a rescaled test function's method, suitably adapted to such nonlocal setting with weak diffusion; they are quite general and also employed to obtain Liouville type results in the whole space.
\end{abstract}

\maketitle 

\allowdisplaybreaks

\section{Introduction}
In this paper we obtain Liouville theorems for supersolutions of  semilinear integral equations in the half-space $\R^N_+:=\left\{x\in\R^N:\,x_N>0\right\}$ with the following class of nonlocal operators
\begin{equation}\label{eq1}
Lu\left(x\right)
=\left(1-s\right)\int_{\R^{N}}\frac{u\left(x+y\right)+u\left(x-y\right)-2u\left(x\right)}{\left|y\right|^{N+2s}}a\left(\frac{y}{\left|y\right|}\right)dy,
\end{equation}
where $0<s<1$ and $a\left(\theta\right)\in L^{\infty}\left(\mathbb{S}^{N-1}\right)$ is a nonnegative and even function on the unit sphere $\mathbb{S}^{N-1}$ of $\R^N$. We will be more precise later, but let us emphasize immediately that the function $a\left(\theta\right)$ can be chosen so that the  operator diffuses only along a cone. 

\smallskip

The main result will be about nonexistence of   classical solutions, besides the trivial one,  $u\in C^{2}\left(\R_{+}^{N}\right)\cap\mathcal{L}_{s}$ of the problem
\begin{align}\label{1-0-3}
\left\lbrace 
\begin{aligned}
-Lu&\geq u^{p},&&x\in\R_{+}^{N},\\
u&\geq0,&&x\in\R^N,
\end{aligned}
\right.
\end{align}
where $\mathcal{L}_{s}$ is the natural functional space required that will be defined precisely later. 

\smallskip

Nonexistence of entire bounded harmonic functions, goes back to Cauchy (1844) and Liouville (1850), but since the acclaimed works of Gidas and Spruck \cite{GS2} much attention has been given to so called semilinear Liouville theorems in the whole space and in half-spaces. Indeed these nonexistence results in unbounded domains correspond to existence results in bounded domains thanks to the priori bounds they imply and the use of Leray-Schauder degree theory. 

A particular role has been given in these decades to determine the range of existence, or better lack of it, when instead of considering solutions, one concentrates on supersolutions. This has been done  for very different cases:  the case of cone-like domains for linear and quasilinear operators appeared for instance in \cite{AS, BVP, BCDN, GMQS,KLM}, existence and nonexistence for positive supersolutions in the whole space, half-spaces and exterior domains for fully nonlinear inequalities can be found e.g. in  \cite{AS2,CL,L}, systems and sub-Laplacians have been respectively addressed in \cite{BM,QS} and \cite{BG,BCDC}, while as far as nonlinear degenerate elliptic operators are concerned we refer to \cite{BC, BGL}.

We cannot recall all the results in this area but the interested readers can refer to the book \cite{QuS} and the references therein. We just remind that, for the Laplacian,  if
$$1\leq p\leq \frac{N+1}{N-1}$$
then the inequality
$$\Delta u+u^p\leq 0\quad\text{in $\R^N_+$}$$
admits only the trivial solution (in fact such result is still true for $p\geq-1$, but in this paper we only consider the superlinear case $p\geq1$). Furthermore, the bound $\frac{N+1}{N-1}$ is sharp in the sense that for $p> \frac{N+1}{N-1}$ it is possible to construct explicit positive supersolutions. 

\smallskip

In the realm of nonlocal operators, the results are more recent but nonetheless really numerous, see e.g. \cite{BoKN,CDaL,F,ZCCY} concerning entire $s$-harmonic functions and \cite{CLO,DDW,FW2,YZ} for positive solutions and supersolutions  of Lane-Emden type equations/systems in $\R^N$ or in exterior domains. Since we are mainly concerned with supersolutions of semilinear nonlocal operators in half-spaces, we will recall results directly related with this cases. In \cite{CFY}, using the method of moving planes in integral forms,  Chen, Fang and Yang have proved that  for  $1<p\leq\frac{N+ 2s}{N-2s}$ any nonnegative and locally bounded distributional  solution of 
\begin{align*}
\left\lbrace 
\begin{aligned}
(-\Delta)^su&=u^{p},&&x\in\R_{+}^{N},\\
u&=0,&&x\in\overline{\R_{-}^{N}},
\end{aligned}
\right.
\end{align*}
where $\R^N_-:=\left\{x\in\R^N:\,x_N<0\right\}$, is identically zero. The result has been extended in \cite{CX} for a larger class of operators. 

Instead, for supersolutions, the proofs and the range for which the validity of the Liouville theorem holds are different and very few results are available in the half-spaces. 
We wish to mention the result for nonlocal fully nonlinear operators by Nornberg, dos Prazeres and Quaas \cite{NdPQ} where they study the existence of fundamental solutions in conical domains for fully  nonlinear  integral operators of Isaacs type and, as application, they obtain Liouville type results in subcritical regimes. In the special case where the cone is the half-space $\R^N_+$ and the diffusion operator is $(-\Delta)^s$ they find in particular that for
$$1\leq p\leq \frac{N+s}{N-s}$$
the only nonnegative (in $\R^N$)  viscosity supersolutions of the problem
\begin{equation*}
(-\Delta)^s u\geq u^p\quad\text{in $\R^N_+$}
\end{equation*}
is the trivial one (even in this case the results is still valid for $p\geq-1$).

\smallskip

In the works mentioned above, even in the nonlinear case, the kernels of operators considered are, up to some constant, bounded above and below by that of the fractional Laplacian. Here instead, the operators $L$ we consider are part of the larger class of infinitesimal generators of stable Lévy processes
\begin{equation}\label{eq2}
Lu\left(x\right)
=(1-s)\int_{\mathbb{S}^{N-1}}\int_{0}^{+\infty}\frac{u\left(x+t\theta\right)+u\left(x-t\theta\right)-2u\left(x\right)}{t^{1+2s}}dtd\mu, 
\end{equation}
where the spectral measure $\mu=\mu\left(\theta\right)$ on $\mathbb{S}^{N-1}$ satisfies the ellipticity conditions
\begin{align*}
0<\lambda\leq\inf_{\nu\in\mathbb{S}^{N-1}}\int_{\mathbb{S}^{N-1}}\left|\nu\cdot\theta\right|^{2s}d\mu\quad\text{and}\quad\int_{\mathbb{S}^{N-1}}d\mu\leq\Lambda<\infty.
\end{align*}
When $\mu$ is absolutely continuous and $d\mu=a\left(\theta\right)d\theta$, then  \eqref{eq1} and \eqref{eq2} coincide.
\newline
At this stage it is worth to point out that the operator \eqref{eq1} converges, as $s\to1^-$, to a second order linear uniformly elliptic operators with constant coefficients, in the sense that for any $u\in C^2_0(\R^N)$, then  
$$
\lim_{s\to1^-}Lu(x)=\sum_{i,j=1}^Na_{ij}\partial^2_{ij}u(x)\qquad\forall x\in\R^N,
$$ 
where the constant coefficients $a_{ij}$ depend on the function $a(\theta)$.

Moreover, let us point out that, since the normalizing constant $1-s$ in \eqref{eq1} is irrelevant for the issues that we address in the present paper, henceforth it will be omitted.

 We refer to the previous works of Ros-Oton and Serra \cite{ROS1}  and in particular to the very interesting paper \cite{ROS2} where, in order to give boundary estimates, some linear Liouville theorems in half-spaces are proved under the further assumption that the function $a(\theta)$ is positive bounded away from zero everywhere.

Contrarily to the cases where  $a(\theta)$ is positive bounded away from zero everywhere, here we only suppose that $a(\theta)$ is positive in some relative open set on $\mathbb S^{N-1}$.  Clearly if $a(\theta)$ is bounded and bounded away from zero then the kernel of the operator \eqref{eq1} can be compared from above and below by that of the fractional Laplacian, which corresponds to the case $a(\theta)$ is a constant function. Other results concerning classifications of entire solutions and H\"older regularity for nonlocal operators with anisotropic diffusion can be found respectively in \cite{FW,KRS}.

\smallskip

We can now give the precise conditions on the function $a(\theta)$. Let us fix the following notation: for any  vector $\nu\in\mathbb{S}^{N-1}$ and $0<\tau\leq1$, we define the closed two fold cone $\Sigma_{\nu,\tau}\left(x\right)$ of aperture $\arccos(1-\tau)\in\left(0,\frac\pi2\right]$, with vertex $x\in\R^{N}$ and axis $\nu$,   by
\begin{align}\label{cone}
\Sigma_{\nu,\tau}\left(x\right):=\left\{y\in\R^{N}:\;\left|\left(y-x\right)\cdot\nu\right|\geq\left(1-\tau\right)\left|y-x\right|\right\}.
\end{align}
In all the paper, we assume that for some constants $0<d<D$,
\begin{align}\label{1-0-1}
0\leq a(\theta)\leq D\quad\text{in $\mathbb{S}^{N-1}$}
\end{align}
 and  that there exist $\nu_{0}\in\mathbb{S}^{N-1}$ and $0<\tau_{0}\leq1$ such that
\begin{align}\label{1-0-2}
a(\theta)\geq d>0\quad\text{in $\Sigma_{\nu_{0},\tau_{0}}\left(0\right)\cap \mathbb{S}^{N-1}$}.
\end{align}
Hence in all the proofs we can only use that the operator \eqref{eq1} diffuses along a fixed cone. This feature will induce some delicate geometric constructions.

\smallskip

Let
\begin{align*}
\mathcal{L}_{s}=\left\{u\in L^{1}_{loc}\left(\R^{N}\right):\,\limsup_{\left|x\right|\rightarrow+\infty}\frac{\left|u\right|}{\left|x\right|^{2s-\delta}}<+\infty\text{\; for some\,}\delta\in(0,2s]\right\}.
\end{align*}

Our main result, is the following
\begin{thm}\label{th1}
Let $s\in(0,1)$ and let $L$ be any operator of the form \eqref{eq1} satisfying \eqref{1-0-1}-\eqref{1-0-2}. If $1\leq p\leq\frac{N+s}{N-s}$ and $u\in C^{2}\left(\R_{+}^{N}\right)\cap\mathcal{L}_{s}$ is a solution of \eqref{1-0-3}, then $u\equiv0$.
\end{thm}
We were inspired by the pioneering proof used by Berestycki, Capuzzo Dolcetta and Nirenberg in \cite{BCDN} for supersolutions of semilinear Liouville  Theorem in cones. But it is important to notice that the key ingredients they use need to be completely reconsidered due  both to the nonlocal character of the operators we consider and their weak diffusion. Even the simple integrations by parts formula,  since the test functions we shall use in the proof of Theorem \ref{th1} are not smooth on $\partial\R^N_+$, needs to be proved (see Proposition \ref{2-4}). The other novelty is in the ad hoc construction of the test functions. Indeed we construct a sequence of nonnegative compactly supported test functions that converge to 1 in the whole half space, but their supports depend on the cone $\Sigma_{\nu_{0},\tau_{0}}\left(0\right)$ where the function $a(\theta)$ is supported. 

Interestingly the bound on $p$, does not depend on the size of the cone. Furthermore, at least in the case of the fractional Laplacian, the bound is optimal: for any $p>\frac{N+s}{N-s}$ the problem
\begin{align*}
\left\lbrace 
\begin{aligned}
\left(-\Delta\right)^{s}u&\geq u^{p},&&x\in\R_{+}^{N},\\
u&=0,&&x\in\R_{-}^{N}
\end{aligned}
\right.
\end{align*}
admits positive solutions $u\in C^{2}\left(\R_{+}^{N}\right)\cap\mathcal{L}_{s}$. This is proved in  Theorem \ref{4-1}. 

\smallskip

The arguments used in the proof of Theorem \ref{th1} are quite flexible and they also apply to get the equivalent Liouville  result in the whole space $\R^N$,  in a simplified form due to the absence of the boundary of the domain which, instead, poses additional difficulties in the case of half-spaces. We obtain nonexistence of positive solutions provided $1\leq p\leq\frac{N}{N-2s}$, see Theorem \ref{5-1}. This range is known to be optimal for the fractional Laplacian and even more, as showed by Felmer and Quaas in \cite{FQ}, for Pucci's nonlinear operators ${\mathcal M}^\pm$, which are extremal operators among the class of linear integral operators whose kernels, differently to the cases treated in the present paper, are comparable to that of the fractional Laplacian. The critical exponents associated to Pucci's operators are $p=\frac{N^\pm}{N^\pm-2s}$, where $N^\pm$ are dimensional-like numbers which reduce to the dimension $N$ when $-{\mathcal M}^\pm=(-\Delta)^s$.

\medskip
The paper is organized as follows. Section \ref{s2} contains several technical results that will be used throughout the paper. Section \ref{s3} is devoted to the proof of Theorem \ref{th1} and Section \ref{s4} is concerned with the optimality of the exponent $p=\frac{N+s}{N-s}$. In the last Section \ref{s5}, we prove a Liouville theorem in the whole space. Let us mention that the case $N=1$ is much simpler and some remarks concerning the results in that case are scattered along the paper (See e.g. Remarks  \ref{N12} and \ref{N13}).

\section{preliminary}\label{s2}

In this section we prove several technical results that will be used in the proof of Theorem \ref{th1}. We start by the following simple lemma, where it is proved that the nonnegativity assumption of $u$ in $\overline{\R^N_-}$ can be, in fact, replaced by the simplest one $u=0$ in $\overline{\R_-^N}$. This reduction is useful in the integration by parts formula of Proposition \ref{2-4}.

\smallskip

Henceforth we denote by $e_N=(0,\ldots,0,1)$ and we use the notation $C=C(\cdot)$ to denote positive constants depending on given quantities.

\begin{lem}\label{2-5}
For any classical solution $u$ of \eqref{1-0-3}, the translation-truncation function (see Figure \ref{figure1}) 
\begin{align}\label{eq10}
\widetilde{u}\left(x\right)=
\left\lbrace 
\begin{aligned}
&u\left(x+e_{N}\right),&&x\in\R_{+}^{N},\\
&0,&&x\in\overline{\R_{-}^{N}}
\end{aligned}
\right.
\end{align}
is a nonnegative classical solution of the problem
\begin{align}\label{2-5-1}
\left\lbrace 
\begin{aligned}
-Lu&\geq u^{p},&&x\in\R_{+}^{N},\\
u&=0,&&x\in\overline{\R_{-}^{N}}.
\end{aligned}
\right.
\end{align}
\begin{figure}[htbp]
\centering
\includegraphics[width=0.6\textwidth]{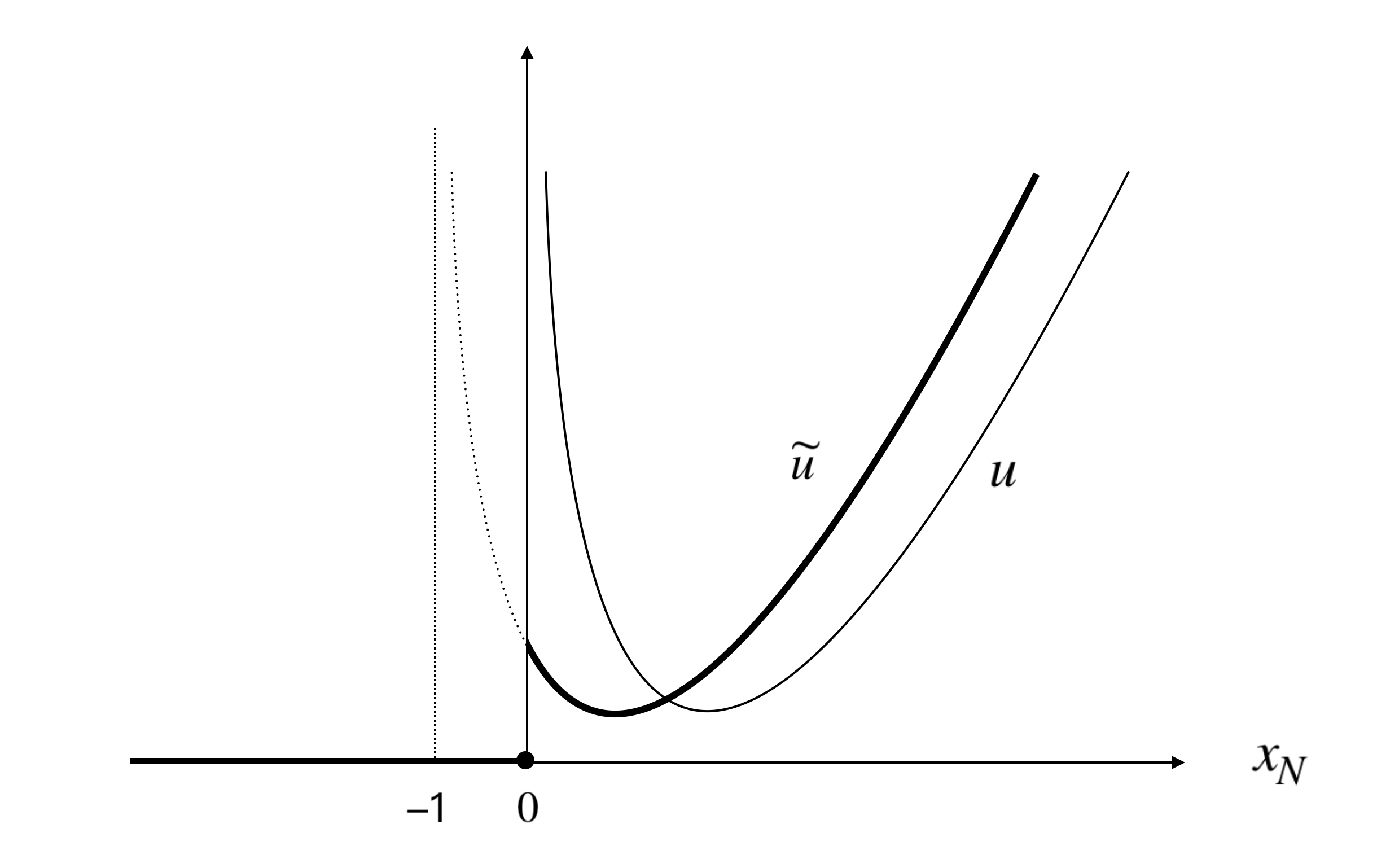}
\caption{The graph of $\widetilde{u}$.}
\label{figure1}
\end{figure}
\end{lem}
\begin{proof}
For $x\in\R_{+}^{N}$ and any $y\in\R^{N}$ it holds that
\begin{align*}
\widetilde{u}\left(x\pm y\right)\leq u\left(x+e_{N}\pm y\right).
\end{align*}
Hence
\begin{align*}
-L\widetilde{u}\left(x\right)
&=-\int_{\R^{N}}\frac{\widetilde{u}\left(x+y\right)+\widetilde{u}\left(x-y\right)-2\widetilde{u}\left(x\right)}{\left|y\right|^{N+2s}}a\left(\frac{y}{\left|y\right|}\right)dy\\
&\geq-\int_{\R^{N}}\frac{u\left(x+e_{N}+y\right)+u\left(x+e_{N}-y\right)-2u\left(x+e_{N}\right)}{\left|y\right|^{N+2s}}a\left(\frac{y}{\left|y\right|}\right)dy\\
&=-Lu\left(x+e_{N}\right).
\end{align*}
Since $u$ is a solution of \eqref{1-0-3}, we have
\begin{align*}
-L\widetilde{u}\left(x\right)\geq-Lu\left(x+e_{N}\right)\geq u^{p}\left(x+e_{N}\right)=\widetilde{u}^{p}\left(x\right).
\end{align*}
\end{proof}

For later purposes we consider the class of nonnegative functions $u\in C^{2}\left(\R_{+}^{N}\right)\cap\mathcal{L}_{s}$ satisfying the following regularity condition up to the boundary of $\R^N_+$:
\begin{align}\label{2-0-1}
\left\|u\right\|_{C^{2}\left(K\cap{\R^{N}_{+}}\right)}<+\infty,\quad \forall \text{ compact }K\subseteq\overline{\R_{+}^{N}}.
\end{align}
Note that if $u$ is a classical solution  of \eqref{1-0-3}, then the function $\widetilde{u}$ defined in Lemma \ref{2-5} by \eqref{eq10}  meets condition \eqref{2-0-1}.


It is well known that $L$ is a self adjoint operator but we need to take care of the “integration by parts”
since the test function is not smooth. This is the object of the next Proposition.

\begin{pro}\label{2-4}
Let $(2s-1)_+<\alpha<2s$, $\psi\in C^{2}_{0}\left(\R^{N}\right)$ and $v_{\alpha}\left(x\right)=\left(x_{N}\right)_{+}^{\alpha}\psi\left(x\right)$. Then for any $u\in C^{2}\left(\R_{+}^{N}\right)\cap\mathcal{L}_{s}$ satisfying \eqref{2-0-1} and $u=0$ in $\overline{\R_{-}^{N}}$, we have
\begin{align}\label{eq1rev}
\int_{\R^{N}}uLv_{\alpha}dx=\int_{\R^{N}}v_{\alpha}Ludx.
\end{align}
\end{pro}
\begin{proof}
Assume that supp $\psi\subseteq B_{R}$, $R>1$. Let us define two maps 
\begin{align*}
F_{1},F_{2}:\R^{N}\times\left(\R^{N}\backslash\left\{0\right\}\right)\rightarrow\R
\end{align*}
by
\begin{align*}
F_{1}\left(x,y\right):=\frac{v_{\alpha}\left(x+y\right)+v_{\alpha}\left(x-y\right)-2v_{\alpha}\left(x\right)}{\left|y\right|^{N+2s}}a\left(\frac{y}{\left|y\right|}\right)u\left(x\right)
\end{align*}
and
\begin{align*}
F_{2}\left(x,y\right):=\frac{u\left(x+y\right)+u\left(x-y\right)-2u\left(x\right)}{\left|y\right|^{N+2s}}a\left(\frac{y}{\left|y\right|}\right)v_{\alpha}\left(x\right).
\end{align*}
What needs to be proved in order to get \eqref{eq1rev} is
\begin{align}\label{2-4-1}
\int_{\R^{N}}\left(\int_{\R^{N}}\left|F_{1}\right|dy\right)dx<+\infty
\end{align}
and
\begin{align}\label{2-4-2}
\int_{\R^{N}}\left(\int_{\R^{N}}\left|F_{2}\right|dy\right)dx<+\infty.
\end{align}
Once  \eqref{2-4-1} and \eqref{2-4-2} are proved i.e. that $F_{1},F_{2}\in L^{1}\left(\R^{N}\times\R^{N}\right)$, we may apply Fubini-Tonelli's theorem and it is classical that \begin{align*}
&\int_{\R^{N}}uLv_{\alpha}dx=\int_{\R^{N}}\left(\int_{\R^{N}}F_{1}dy\right)dx=\int_{\R^{N}}\left(\int_{\R^{N}}F_{1}dx\right)dy\\
&=\int_{\R^{N}}\left(\int_{\R^{N}}\left[v_{\alpha}\left(x+y\right)+v_{\alpha}\left(x-y\right)-2v_{\alpha}\left(x\right)\right]u\left(x\right)dx\right)\frac{1}{\left|y\right|^{N+2s}}a\left(\frac{y}{\left|y\right|}\right)dy\\
&=\int_{\R^{N}}\left(\int_{\R^{N}}\left[u\left(x+y\right)+u\left(x-y\right)-2u\left(x\right)\right]v_{\alpha}\left(x\right)dx\right)\frac{1}{\left|y\right|^{N+2s}}a\left(\frac{y}{\left|y\right|}\right)dy\\
&=\int_{\R^{N}}\left(\int_{\R^{N}}F_{2}dx\right)dy=\int_{\R^{N}}\left(\int_{\R^{N}}F_{2}dy\right)dx=\int_{\R^{N}}v_{\alpha}Ludx.
\end{align*}

We first prove \eqref{2-4-1}. Note that
\begin{align}\label{2-4-3}
F_{1}\left(x,y\right)=0\quad\forall\left(x,y\right)\in\overline{\R_{-}^{N}}\times\left(\R^{N}\backslash\left\{0\right\}\right),
\end{align}
we only consider $\left(x,y\right)\in\R_{+}^{N}\times\left(\R^{N}\backslash\left\{0\right\}\right)$. For any $x\in\R^{N}_{+}\cap \left\{\left|x\right|\geq2R\right\}$, there is $v_{\alpha}\left(x\right)=0$. Moreover, if $\left|y\pm x\right|\geq R$, then $v_{\alpha}\left(x\pm y\right)=0$, while if $\left|y\pm x\right|\leq R$, we have
\begin{align*}
\left|y\right|\geq\left|x\right|-\left|y-x\right|\geq\frac{\left|x\right|}{2}.
\end{align*}
 Therefore, for any $x\in\R^{N}_{+}\cap \left\{\left|x\right|\geq2R\right\}$,
\begin{align*}
&\int_{\R^{N}}\frac{\left|v_{\alpha}\left(x+y\right)+v_{\alpha}\left(x-y\right)-2v_{\alpha}\left(x\right)\right|}{\left|y\right|^{N+2s}}dy\\
&=\int_{B_{R}\left(x\right)}\frac{\left|v_{\alpha}\left(x-y\right)\right|}{\left|y\right|^{N+2s}}dy+\int_{B_{R}\left(-x\right)}\frac{\left|v_{\alpha}\left(x+y\right)\right|}{\left|y\right|^{N+2s}}dy\\
&=\int_{B_{R}\left(x\right)}\frac{\left|\left(x_{N}-y_{N}\right)_{+}^{\alpha}\psi\left(x-y\right)\right|}{\left|y\right|^{N+2s}}dy+\int_{B_{R}\left(-x\right)}\frac{\left|\left(x_{N}+y_{N}\right)_{+}^{\alpha}\psi\left(x+y\right)\right|}{\left|y\right|^{N+2s}}dy\\
&\leq 2^{N+2s}R^{\alpha}\left\|\psi\right\|_{L^{\infty}\left(\R^{N}\right)}\left(\left|B_{R}\left(x\right)\right|+\left|B_{R}\left(-x\right)\right|\right)\frac{1}{\left|x\right|^{N+2s}}\\
&=C\left(N,s,\alpha,R,\left\|\psi\right\|_{L^{\infty}\left(\R^{N}\right)}\right)\frac{1}{\left|x\right|^{N+2s}}.
\end{align*}
For any $x\in\R^{N}_{+}\cap B_{2R}$, letting $z=\frac{y}{x_{N}}$, we have
\begin{align*}
&\int_{\R^{N}}\frac{\left|v_{\alpha}\left(x+y\right)+v_{\alpha}\left(x-y\right)-2v_{\alpha}\left(x\right)\right|}{\left|y\right|^{N+2s}}dy\\
&=\int_{\R^{N}}\frac{\left|\left(x_{N}+y_{N}\right)_{+}^{\alpha}\psi\left(x+y\right)+\left(x_{N}-y_{N}\right)_{+}^{\alpha}\psi\left(x-y\right)-2x_{N}^{\alpha}\psi\left(x\right)\right|}{\left|y\right|^{N+2s}}dy\\
&=x_{N}^{\alpha}\int_{\R^{N}}\frac{\left|\left(1+\frac{y_{N}}{x_{N}}\right)_{+}^{\alpha}\psi\left(x+y\right)+\left(1-\frac{y_{N}}{x_{N}}\right)_{+}^{\alpha}\psi\left(x-y\right)-2\psi\left(x\right)\right|}{\left|y\right|^{N+2s}}dy\\
&=\left(\int_{\R^{N}}\frac{\left|A\left(x,z\right)\right|}{\left|z\right|^{N+2s}}dz\right)x_{N}^{\alpha-2s},
\end{align*}
where
\begin{align*}
A\left(x,z\right)=\left(1+z_{N}\right)_{+}^{\alpha}\psi\left(x+x_{N}z\right)+\left(1-z_{N}\right)_{+}^{\alpha}\psi\left(x-x_{N}z\right)-2\psi\left(x\right).
\end{align*}
Since $A$ satisfies
\begin{align*}
\left\lbrace 
\begin{aligned}
&A\left(x,0\right)=0&&\forall x\in\R^N,\\
&A\left(x,z\right)=A\left(x,-z\right)&&\forall\left(x,z\right)\in\R^{N}\times\R^{N},\\
&A\in C^{2}\left(\R^{N}\times\overline{B_{\frac{1}{2}}}\right),
\end{aligned}
\right.
\end{align*}
and $\alpha>0$, then for any $x\in\R^{N}_{+}\cap B_{2R}$ a second order Taylor expansion yields
\begin{align}\label{eq6}
\left|A\right|
\leq C\left(\alpha,R,\left\|\psi\right\|_{C^{2}\left(\R^{N}\right)}\right)\times
\left\lbrace 
\begin{aligned}
&\left|z\right|^{2},&&\left|z\right|<\frac{1}{2},\\
&\left|z\right|^{\alpha},&&\left|z\right|\geq\frac{1}{2}.
\end{aligned}
\right.
\end{align}
Therefore, for $x\in\R^{N}_{+}\cap B_{2R}$, by the condition $\alpha<2s$,
\begin{align}\label{eq7}
\int_{\R^{N}}\frac{\left|A\right|}{\left|z\right|^{N+2s}}dz\leq C\left(N,s,\alpha,R,\left\|\psi\right\|_{C^{2}\left(\R^{N}\right)}\right).
\end{align}
We conclude that
\begin{align}\label{2-4-4}
\nonumber&\int_{\R^{N}}\frac{\left|v_{\alpha}\left(x+y\right)+v_{\alpha}\left(x-y\right)-2v_{\alpha}\left(x\right)\right|}{\left|y\right|^{N+2s}}dy\\
&\leq C\left(N,s,\alpha,R,\left\|\psi\right\|_{C^{2}\left(\R^{N}\right)}\right)\times
\left\lbrace 
\begin{aligned}
&\frac{1}{x_{N}^{2s-\alpha}},&&x\in\R^{N}_{+}\cap B_{2R},\\
&\frac{1}{\left|x\right|^{N+2s}},&&x\in\R^{N}_{+}\cap \left\{\left|x\right|\geq2R\right\}.
\end{aligned}
\right.
\end{align}
Let us notice that by the assumption $u\in\mathcal{L}_{s}$, $u$ satisfies \eqref{2-0-1} and $u=0$ in $\overline{\R_{-}^{N}}$, we deduce that there exists $\beta=\beta\left(u\right)>0$ such that
\begin{align}\label{2-4-5}
\left|u\right|\leq \beta\left(1+\left|x\right|^{2s-\delta}\right)\quad\forall x\in\R^{N}.
\end{align}
Since $\alpha>2s-1$, by \eqref{1-0-1}, \eqref{2-4-3}, \eqref{2-4-4} and \eqref{2-4-5},
\begin{align*}
&\int_{\R^{N}}\left(\int_{\R^{N}}\left|F_{1}\right|dy\right)dx\\
&=\int_{\R_{+}^{N}}\left|u\right|\left(\int_{\R^{N}}\frac{\left|v_{\alpha}\left(x+y\right)+v_{\alpha}\left(x-y\right)-2v_{\alpha}\left(x\right)\right|}{\left|y\right|^{N+2s}}\left|a\left(\frac{y}{\left|y\right|}\right)\right|dy\right)dx\\
&\leq D\int_{\R_{+}^{N}}\left|u\right|\left(\int_{\R^{N}}\frac{\left|v_{\alpha}\left(x+y\right)+v_{\alpha}\left(x-y\right)-2v_{\alpha}\left(x\right)\right|}{\left|y\right|^{N+2s}}dy\right)dx\\
&\leq C\left(D,N,s,\alpha,R,\left\|\psi\right\|_{C^{2}\left(\R^{N}\right)}\right)\left[\int_{\R^{N}_{+}\cap B_{2R}}\frac{\left|u\right|}{x_{N}^{2s-\alpha}}dx+\int_{\R^{N}_{+}\cap\left\{\left|x\right|\geq2R\right\}}\frac{\left|u\right|}{\left|x\right|^{N+2s}}dx\right]\\
&\leq C\left(D,N,s,\alpha,R,\left\|\psi\right\|_{C^{2}\left(\R^{N}\right)}\right)\left[\left\|u\right\|_{L^{\infty}\left(B_{2R}\right)}\int_{\R^{N}_{+}\cap B_{2R}}\frac{1}{x_{N}^{2s-\alpha}}dx\right.\\
&\quad\left.+\beta\int_{\R^{N}_{+}\cap\left\{\left|x\right|\geq2R\right\}}\frac{1+\left|x\right|^{2s-\delta}}{\left|x\right|^{N+2s}}dx\right]<+\infty.
\end{align*}
Next, we prove \eqref{2-4-2}. Since

\begin{align}\label{2-4-6}
F_{2}\left(x,y\right)=0\quad\forall\left(x,y\right)\in\left(\overline{\R^{N}_{-}}\cup\left\{\left|x\right|\geq R\right\}\right)\times\left(\R^{N}\backslash\left\{0\right\}\right),
\end{align}
we need to focus on $\left(x,y\right)\in\left(\R_{+}^{N}\cap B_{R}\right)\times\left(\R^{N}\backslash\left\{0\right\}\right)$. For any $x\in\R^{N}_{+}\cap B_{R}$, by \eqref{2-4-5},
\begin{align}\label{2-4-7}
\left|u\left(x\pm y\right)\right|\leq \beta\left(1+\left|x\pm y\right|^{2s-\delta}\right)\leq C\left(s,\delta,\beta,R\right)\left(1+\left|y\right|^{2s-\delta}\right)\quad\forall y\in\R^{N}.
\end{align}
Moreover, by \eqref{2-0-1}, we obtain $u\in C^{2}\left(\overline{\R^{N}_{+}}\cap B_{2R}\right)$ and together with \eqref{2-4-7}, we infer that
\begin{align*}
\frac{\left|u\left(x+y\right)+u\left(x-y\right)-2u\left(x\right)\right|}{\left|y\right|^{N+2s}}
\leq
\left\lbrace 
\begin{aligned}
&\left\|u\right\|_{C^{2}\left(B_{x_{N}}\left(x\right)\right)}\frac{1}{\left|y\right|^{N+2s-2}},&&\left|y\right|<x_{N},\\
&C\left(s,\delta,\beta,R\right)\frac{1+\left|y\right|^{2s-\delta}}{\left|y\right|^{N+2s}},&&\left|y\right|\geq x_{N}.
\end{aligned}
\right.
\end{align*}
Consequently, for any $x\in\R^{N}_{+}\cap B_{R}$, we get $\left\|u\right\|_{C^{2}\left(B_{x_{N}}\left(x\right)\right)}\leq\left\|u\right\|_{C^{2}\left(\overline{\R^{N}_{+}}\cap B_{2R}\right)}$ and then
\begin{align}\label{2-4-8}
\nonumber&\int_{\R^{N}}\frac{\left|u\left(x+y\right)+u\left(x-y\right)-2u\left(x\right)\right|}{\left|y\right|^{N+2s}}dy\\
\nonumber&\leq\left\|u\right\|_{C^{2}\left(\overline{\R^{N}_{+}}\cap B_{2R}\right)}\int_{\left|y\right|<x_{N}}\frac{1}{\left|y\right|^{N+2s-2}}dy+C\left(s,\delta,\beta,R\right)\int_{\left|y\right|\geq x_{N}}\frac{1+\left|y\right|^{2s-\delta}}{\left|y\right|^{N+2s}}dy\\
\nonumber&\leq C\left(s,\delta,\beta,R,\left\|u\right\|_{C^{2}\left(\overline{\R^{N}_{+}}\cap B_{2R}\right)}\right)\left(x_{N}^{2-2s}+\frac{1}{x_{N}^{2s}}+\frac{1}{x_{N}^{\delta}}\right)\\
&\leq C\left(s,\delta,\beta,R,\left\|u\right\|_{C^{2}\left(\overline{\R^{N}_{+}}\cap B_{2R}\right)}\right)\frac{1}{x_{N}^{2s}}.
\end{align}
Reasoning that $\alpha>2s-1$, by \eqref{1-0-1}, \eqref{2-4-6} and \eqref{2-4-8},
\begin{align*}
&\int_{\R^{N}}\left(\int_{\R^{N}}\left|F_{2}\right|dy\right)dx\\
&=\int_{\R^{N}_{+}\cap B_{R}}x_{N}^{\alpha}\left|\psi\right|\left(\int_{\R^{N}}\frac{\left|u\left(x+y\right)+u\left(x-y\right)-2u\left(x\right)\right|}{\left|y\right|^{N+2s}}\left|a\left(\frac{y}{\left|y\right|}\right)\right|dy\right)dx\\
&\leq D\left\|\psi\right\|_{L^{\infty}\left(\R^{N}\right)}\int_{\R^{N}_{+}\cap B_{R}}x_{N}^{\alpha}\left(\int_{\R^{N}}\frac{\left|u\left(x+y\right)+u\left(x-y\right)-2u\left(x\right)\right|}{\left|y\right|^{N+2s}}dy\right)dx\\
&\leq C\left(D,s,\delta,\beta,R,\left\|\psi\right\|_{L^{\infty}\left(\R^{N}\right)},\left\|u\right\|_{C^{2}\left(\overline{\R^{N}_{+}}\cap B_{2R}\right)}\right)\int_{\R^{N}_{+}\cap B_{R}}\frac{1}{x_{N}^{2s-\alpha}}dx<+\infty.
\end{align*}
\end{proof}

\begin{rem}
For the sake of completeness we point out that the assumption $\alpha<2s$ in Proposition \ref{2-4} can be in fact removed. Such condition has been used in \eqref{eq6}-\eqref{eq7}. On the other hand, since the function $\psi$ is compactly supported in the ball $B_R$, a more precise estimate of \eqref{eq6} can be obtained. Indeed, since for any $x\in\R^{N}_{+}\cap B_{2R}$  and $|z|\geq\frac{3R}{x_N}$ we have that $\psi(x\pm x_Nz)=0$, then 
\begin{align*}
\left|A\right|
\leq C\left(\alpha,R,\left\|\psi\right\|_{C^{2}\left(\R^{N}\right)}\right)\times
\left\lbrace 
\begin{aligned}
&\left|z\right|^{2},&&\left|z\right|<\frac{1}{2},\\
&\left|z\right|^{\alpha},&&\frac{1}{2}\leq\left|z\right|<\frac{3R}{x_{N}},\\
&1,&&\left|z\right|\geq\frac{3R}{x_{N}},
\end{aligned}
\right.
\end{align*}
and
\begin{align*}
\int_{\R^{N}}\frac{\left|A\right|}{\left|z\right|^{N+2s}}dz
\leq C\left(N,s,\alpha,R,\left\|\psi\right\|_{C^{2}\left(\R^{N}\right)}\right)\times
\left\lbrace 
\begin{aligned}
&1+\frac{1}{x_{N}^{\alpha-2s}},&&\alpha>2s,\\
&1+\left|\log{x_N}\right|,&&\alpha=2s,\\
&1,&&\alpha<2s.
\end{aligned}
\right.
\end{align*}
Therefore for any $x\in\R^{N}_{+}\cap B_{2R}$,
\begin{align*}
&\int_{\R^{N}}\frac{\left|v_{\alpha}\left(x+y\right)+v_{\alpha}\left(x-y\right)-2v_{\alpha}\left(x\right)\right|}{\left|y\right|^{N+2s}}dy\\
&=\left(\int_{\R^{N}}\frac{\left|A\right|}{\left|z\right|^{N+2s}}dz\right)x_{N}^{\alpha-2s}\\
&\leq C\left(N,s,\alpha,R,\left\|\psi\right\|_{C^{2}\left(\R^{N}\right)}\right)\times
\left\lbrace 
\begin{aligned}
&1,&&\alpha>2s,\\
&1+\left|\log{x_N}\right|,&&\alpha=2s,\\
&\frac{1}{x_{N}^{2s-\alpha}},&&\alpha<2s
\end{aligned}
\right.
\end{align*}
and as a result
\begin{align*}
&\int_{\R^{N}_{+}\cap B_{2R}}\left(\int_{\R^{N}}\left|F_{1}\right|dy\right)dx\\
&\leq D\left\|u\right\|_{L^{\infty}\left(B_{2R}\right)}\int_{\R^{N}_{+}\cap B_{2R}}\left(\int_{\R^{N}}\frac{\left|v_{\alpha}\left(x+y\right)+v_{\alpha}\left(x-y\right)-2v_{\alpha}\left(x\right)\right|}{\left|y\right|^{N+2s}}dy\right)dx<+\infty.
\end{align*}
\end{rem}

We are now concerned with the computation of the operator \eqref{eq1} acting on barrier type functions. The result of the next Lemma is partially known, in particular the fact that for $\alpha=s$, $\left(x_{N}\right)_{+}^{\alpha}$ is harmonic  see e.g. \cite{BV,ROS2}. But a complete result for $\alpha\in (0,2s)$ doesn't seem to be available, so we decided to put here the statement and a short proof.

\begin{lem}\label{2-3}
Let $0<\alpha<2s$ and $w_{\alpha}\left(x\right)=\left(x_{N}\right)_{+}^{\alpha}$. For any $x\in\R^{N}_{+}$,
\begin{align*}
Lw_{\alpha}\left(x\right)=C_{\alpha}x_{N}^{\alpha-2s},
\end{align*}
where
\begin{align*}
C_{\alpha}
\left\lbrace 
\begin{aligned}
&<0,&&0<\alpha<s,\\
&=0,&&\alpha=s,\\
&>0,&&s<\alpha<2s.
\end{aligned}
\right.
\end{align*}
\end{lem}

\begin{proof}
For any $x\in\R^{N}_{+}$,
\begin{align*}
Lw_{\alpha}\left(x\right)
&=\int_{\R^{N}}\frac{\left(x_{N}+y_{N}\right)_{+}^{\alpha}+\left(x_{N}-y_{N}\right)_{+}^{\alpha}-2x_{N}^{\alpha}}{\left|y\right|^{N+2s}}a\left(\frac{y}{\left|y\right|}\right)dy\\
&=\int_{\mathbb{S}^{N-1}}\left(\int_{0}^{+\infty}\frac{\left(x_{N}+r\theta_{N}\right)_{+}^{\alpha}+\left(x_{N}-r\theta_{N}\right)_{+}^{\alpha}-2x_{N}^{\alpha}}{r^{1+2s}}dr\right)a\left(\theta\right)d\theta\\
&=x_{N}^{\alpha}\int_{\mathbb{S}^{N-1}}\left(\int_{0}^{+\infty}\frac{\left(1+\frac{\theta_{N}}{x_{N}}r\right)_{+}^{\alpha}+\left(1-\frac{\theta_{N}}{x_{N}}r\right)_{+}^{\alpha}-2}{r^{1+2s}}dr\right)a\left(\theta\right)d\theta.
\end{align*}
By the change of variable $t=\frac{\left|\theta_{N}\right|}{x_{N}}r$, with $\theta_{N}\neq0$, we obtain
\begin{align*}
&\int_{0}^{+\infty}\frac{\left(1+\frac{\theta_{N}}{x_{N}}r\right)_{+}^{\alpha}+\left(1-\frac{\theta_{N}}{x_{N}}r\right)_{+}^{\alpha}-2}{r^{1+2s}}dr\\
&=\left(\frac{\left|\theta_{N}\right|}{x_{N}}\right)^{2s}\int_{0}^{+\infty}\frac{\left(1+t\right)_{+}^{\alpha}+\left(1-t\right)_{+}^{\alpha}-2}{t^{1+2s}}dt\\
&=:c_{\alpha}\left|\theta_{N}\right|^{2s}x_{N}^{-2s}.
\end{align*}
Moreover, the above equalities are still true for $\theta_{N}=0$.
Thus we have
\begin{align*}
Lw_{\alpha}\left(x\right)=C_\alpha x_{N}^{\alpha-2s}
\end{align*}
 with
$$
C_\alpha:=c_{\alpha}\left(\int_{\mathbb{S}^{N-1}}\left|\theta_{N}\right|^{2s}a\left(\theta\right)d\theta\right).
$$
By the assumption \eqref{1-0-2} it turns out that 
$$
\int_{\mathbb{S}^{N-1}}\left|\theta_{N}\right|^{2s}a\left(\theta\right)d\theta>0.
$$
Hence the sign of $C_\alpha$ is given by $c_\alpha$. For this, by the same arguments of \cite[Lemma 2.4]{BV}  concerning the case $\alpha=s$ (see also \cite[Lemma 2.3]{BGS}), we get that
\begin{align*}
c_{\alpha}
\left\lbrace 
\begin{aligned}
&<0,&&0<\alpha<s,\\
&=0,&&\alpha=s,\\
&>0,&&s<\alpha<2s.
\end{aligned}
\right.
\end{align*}
Hence the result follows.
\end{proof}

\begin{lem}\label{2-1}
Let $u,g,h\in C^{2}\left(\R_{+}^{N}\right)\cap\mathcal{L}_{s}$.
\begin{itemize}
\item[(1)] For any $R>0$ and $u_{R}=u\left(\frac{x}{R}\right)$,
\begin{align*}
Lu_{R}\left(x\right)
=R^{-2s}Lu\left(\frac{x}{R}\right)\quad\forall x\in\R_{+}^{N}.
\end{align*}
\item[(2)] For $u=gh$,
\begin{align*}
Lu\left(x\right)
=g\left(x\right)Lh\left(x\right)+h\left(x\right)Lg\left(x\right)+l\left[g,h\right]\left(x\right)\quad\forall x\in\R_{+}^{N},
\end{align*}
where
\begin{align*}
l\left[g,h\right]\left(x\right)=&\int_{\R^{N}}\left\{\frac{\left[g\left(x+y\right)-g\left(x\right)\right]\left[h\left(x+y\right)-h\left(x\right)\right]}{\left|y\right|^{N+2s}}\right.\\
&\left.+\,\frac{\left[g\left(x-y\right)-g\left(x\right)\right]\left[h\left(x-y\right)-h\left(x\right)\right]}{\left|y\right|^{N+2s}}\right\}a\left(\frac{y}{\left|y\right|}\right)dy.
\end{align*}
\end{itemize}
\end{lem}
\begin{proof}
For any $R>0$,
\begin{align*}
Lu_{R}\left(x\right)
&=\int_{\R^{N}}\frac{u\left(\frac{x+y}{R}\right)+u\left(\frac{x-y}{R}\right)-2u\left(\frac{x}{R}\right)}{\left|y\right|^{N+2s}}a\left(\frac{y}{\left|y\right|}\right)dy\\
&=R^{-2s}\int_{\R^{N}}\frac{u\left(\frac{x}{R}+y\right)+u\left(\frac{x}{R}-y\right)-2u\left(\frac{x}{R}\right)}{\left|y\right|^{N+2s}}a\left(\frac{y}{\left|y\right|}\right)dy\\
&=R^{-2s}Lu\left(\frac{x}{R}\right),
\end{align*}
hence we get Lemma \ref{2-1}-(1). Next, for $x\in\R_{+}^{N}$ and any $y\in\R^{N}$, 
\begin{align*}
g&\left(x+y\right)h\left(x+y\right)+g\left(x-y\right)h\left(x-y\right)-2g\left(x\right)h\left(x\right)\\
&=\left[g\left(x+y\right)+g\left(x-y\right)-2g\left(x\right)\right]h\left(x\right)+\left[h\left(x+y\right)+h\left(x-y\right)-2h\left(x\right)\right]g\left(x\right)\\
&\quad+\left[g\left(x+y\right)-g\left(x\right)\right]\left[h\left(x+y\right)-h\left(x\right)\right]+\left[g\left(x-y\right)-g\left(x\right)\right]\left[h\left(x-y\right)-h\left(x\right)\right].
\end{align*}
Lemma \ref{2-1}-(2) is a direct result of the above identity.
\end{proof}

We conclude this section with a continuity property of the function $l$ introduced in Lemma \ref{2-1}-(2) that will be used in Theorem \ref{th1}.

\begin{lem}\label{3-2}
Let $(2s-1)_+<\alpha<\min\left\{1,2s\right\}$, $\varphi\in C^{1}_{0}\left(\R^{N}\right)$ and $w_{\alpha}\left(x\right)=\left(x_{N}\right)_{+}^{\alpha}$. For any $\left\{x_{n}\right\}\subset\R_{+}^{N}$ and $x_{n}\rightarrow x_{0}\in\overline{\R_{+}^{N}}$ as $n\rightarrow+\infty$, we have
\begin{align*}
\lim_{n\rightarrow+\infty}l\left[w_{\alpha},\varphi\right]\left(x_{n}\right)
=l\left[w_{\alpha},\varphi\right]\left(x_{0}\right).
\end{align*}
\end{lem}
\begin{proof}
By Lemma \ref{2-1}-(2),
\begin{align*}
l\left[w_{\alpha},\varphi\right]\left(x_{n}\right)
&=\int_{\R^{N}}\left\{\frac{\left[\left(\left(x_{n}\right)_{N}+y_{N}\right)_{+}^{\alpha}-\left(x_{n}\right)_{N}^{\alpha}\right]\left[\varphi\left(x_{n}+y\right)-\varphi\left(x_{n}\right)\right]}{\left|y\right|^{N+2s}}\right.\\
&\quad\left.+\,\frac{\left[\left(\left(x_{n}\right)_{N}-y_{N}\right)_{+}^{\alpha}-\left(x_{n}\right)_{N}^{\alpha}\right]\left[\varphi\left(x_{n}-y\right)-\varphi\left(x_{n}\right)\right]}{\left|y\right|^{N+2s}}\right\}a\left(\frac{y}{\left|y\right|}\right)dy\\
&:=\int_{\R^{N}}H_{n}\left(y\right)dy.
\end{align*}

Since $0<\alpha<1$, for any $n\in\N$ and $y\in\R^N$ we have
\begin{align*}
\left|\left(\left(x_{n}\right)_{N}\pm y_{N}\right)_{+}^{\alpha}-\left(x_{n}\right)_{N}^{\alpha}\right|\leq \left|y_{N}\right|^{\alpha}.
\end{align*}
Moreover,
\begin{align*}
\left|\varphi\left(x_{n}\pm y\right)-\varphi\left(x_{n}\right)\right|\leq C\left(\left\|\varphi\right\|_{C^{1}\left(\R^{N}\right)}\right)
\times
\left\lbrace 
\begin{aligned}
&|y|,&&\left|y\right|<1,\\
&1,&&\left|y\right|\geq1.
\end{aligned}
\right.
\end{align*}

Hence we get
\begin{align*}
\left|H_{n}\left(y\right)\right|
\leq H\left(y\right)=C\left(D,\left\|\varphi\right\|_{C^{1}\left(\R^{N}\right)}\right)\times
\left\lbrace 
\begin{aligned}
&\frac{1}{\left|y\right|^{N+2s-\alpha-1}},&&\left|y\right|<1,\\
&\frac{1}{\left|y\right|^{N+2s-\alpha}},&&\left|y\right|\geq1.
\end{aligned}
\right.
\end{align*}
By the assumption $2s-1<\alpha<2s$, we have $H\in L^{1}\left(\R^{N}\right)$. By Lebesgue's dominated convergence theorem, we reach the conclusion.
\end{proof}

\section{Nonexistence in the half space \texorpdfstring{$\R_{+}^{N}$}{Lg}}\label{s3}


This section is devoted to the proof of Theorem \ref{th1}. Observe that the result and the proof hold true for $N \geq 1$, but it will be assumed $N\geq 2$, the case $N = 1$ being much simpler. Indeed, the operators \eqref{eq1} satisfying \eqref{1-0-1}-\eqref{1-0-2} in dimension $N=1$ reduce, up to a positive multiplicative constant, to the $1$-dimensional fractional Laplacian and the corresponding proof of Theorem \ref{th1} can be done in a similar way to the case $N\geq 2$, but every step is much simpler, details are left to the reader.

 We start with some geometric considerations concerning the cones $\Sigma_{\nu,\tau}\left(x\right)$, see \eqref{cone} for their definition. Next lemma is basic, we report a proof for the reader's convenience.

\begin{lem}\label{3-1}
Given $\nu\in{\mathbb S}^{N-1}$ and $\tau\in(0,1]$, then for any $x\in\partial B_{1}$ it holds
\begin{align*}
\left|\Sigma_{\nu,\tau}\left(x\right)\cap B_{1}\right|>0.
\end{align*}
\begin{proof}
If $x\cdot\nu\neq0$, there exists $t_{0}\in\R$ and $\varepsilon_{0}>0$ such that $y_{0}=x+t_{0}\nu\in\Sigma_{\nu,\tau}\left(x\right)\cap B_{1}$ and $B_{\varepsilon_{0}}\left(y_{0}\right)\subset\left(\Sigma_{\nu,\tau}\left(x\right)\cap B_{1}\right)$, hence
\begin{align*}
\left|\Sigma_{\nu,\tau}\left(x\right)\cap B_{1}\right|>\left|B_{\varepsilon_{0}}\left(y_{0}\right)\right|>0.
\end{align*}
If $x\cdot\nu=0$, then we can pick $\widetilde{\nu}\in\mathbb{S}^{N-1}$ such that $x\cdot\widetilde{\nu}\neq0$ and $0<\left|\widetilde{\nu}-\nu\right|\leq\frac{\tau}{2}$. For any $y\in\Sigma_{\widetilde{\nu},\frac{\tau}{2}}\left(x\right)$ we have
\begin{align*}
\left|\left(y-x\right)\cdot\nu\right|
&\geq\left|\left(y-x\right)\cdot\widetilde{\nu}\right|-\left|\left(y-x\right)\cdot\left(\widetilde{\nu}-\nu\right)\right|\\
&\geq\left(1-\frac{\tau}{2}\right)\left|y-x\right|-\frac{\tau}{2}\left|y-x\right|=\left(1-\tau\right)\left|y-x\right|.
\end{align*}
Then $\Sigma_{\widetilde{\nu},\frac{\tau}{2}}\left(x\right)\subseteq\Sigma_{\nu,\tau}\left(x\right)$ and by the first part of the proof we conclude that 
\begin{align*}
\left|\Sigma_{\nu,\tau}\left(x\right)\cap B_{1}\right|\geq\left|\Sigma_{\widetilde{\nu},\frac{\tau}{2}}\left(x\right)\cap B_{1}\right|>0.
\end{align*}
\end{proof}
\end{lem}

\begin{rem}
In the whole space $\R^{N}$, for any $x\in\partial B_{1}$, it is obvious that 
$$\left|\Sigma_{\nu,\tau}\left(x\right)\cap B_{1}\cap\R^N\right|=\left|\Sigma_{\nu,\tau}\left(x\right)\cap B_{1}\right|>0$$ 
as showed in Lemma \ref{3-1}, see also Figure \ref{figure2}. 
\begin{figure}[htbp]
\centering
\includegraphics[width=0.5\textwidth]{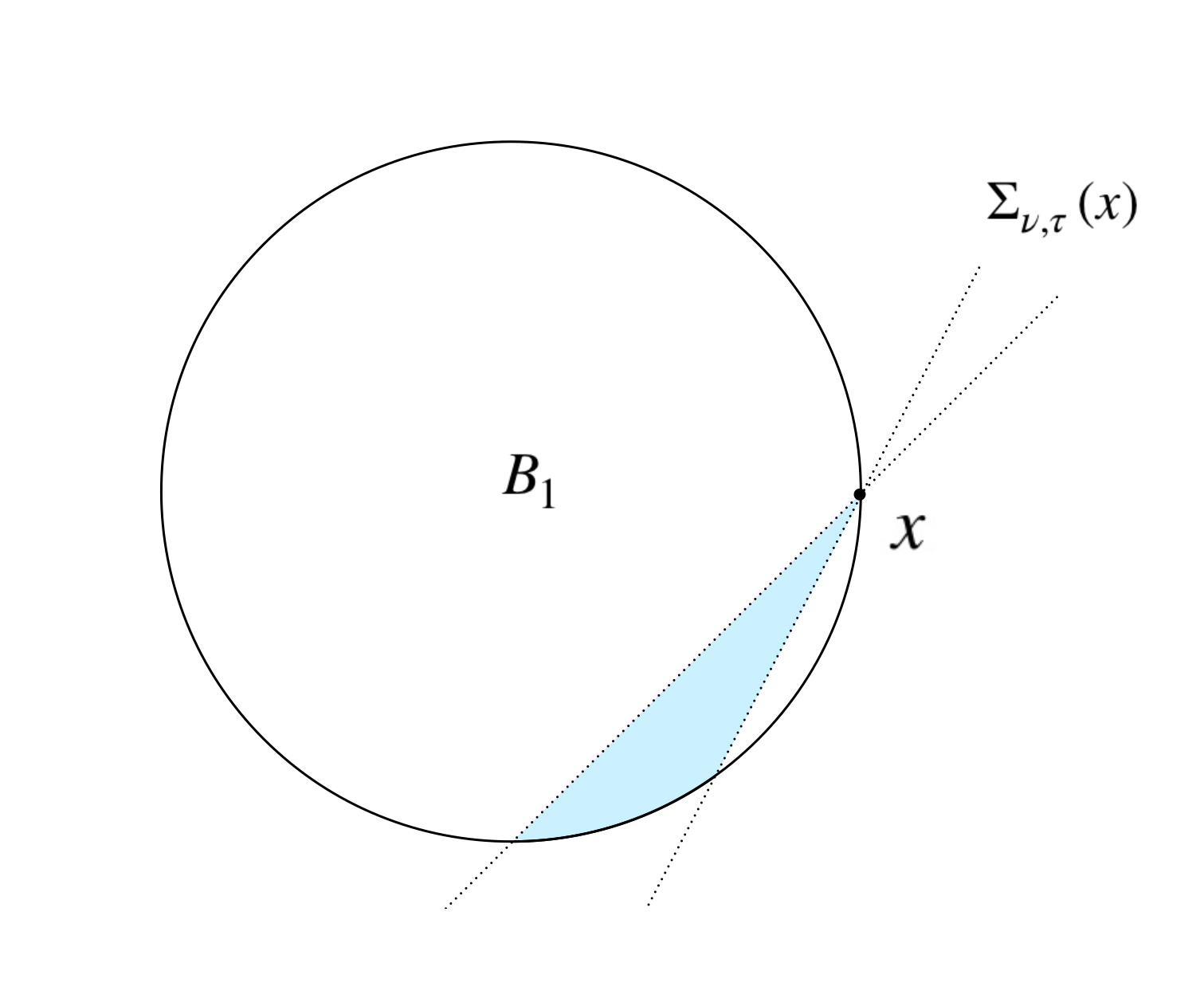}
\caption{The blue area $\left|\Sigma_{\nu,\tau}\left(x\right)\cap B_{1}\right|>0$.}
\label{figure2}
\end{figure}

However, in the half space $\R_{+}^{N}$, there might be that for some $x\in\partial\R_{+}^{N}\cap \partial B_{1}$
\begin{align*}
\left|\Sigma_{\nu,\tau}\left(x\right)\cap B_{1}\cap\R^{N}_{+}\right|=0,
\end{align*}
see Figure \ref{figure3}.
\newline
Our strategy is therefore the following: since, as a consequence of Lemma \ref{3-1}, $$\left|\Sigma_{\nu,\tau}\left(0\right)\cap B_{1}\left(e_{N}\right)\right|>0,$$ we can then slightly move the center of the ball $B_{1}\left(e_{N}\right)$ from $e_{N}=(0,\ldots,0,1)$ to $\left(1-\gamma\right)e_{N}=\left(0,...,0,1-\gamma\right)$, with $\gamma>0$ sufficiently small, in order to guarantee that
\begin{align*}
\left|B_{1}\left(\left(1-\gamma\right)e_{N}\right)\cap{\R^{N}_{-}}\right|>0
\end{align*}
and that for any $x\in\partial\R_{+}^{N}\cap \partial B_{1}\left(\left(1-\gamma\right)e_{N}\right)$
\begin{align}\label{eq8}
\left|\Sigma_{\nu,\tau}\left(x\right)\cap B_{1}\left(\left(1-\gamma\right)e_{N}\right)\cap\R^{N}_{+}\right|>0,
\end{align}
see Figure \ref{figure4}. In fact,  condition \eqref{eq8} is still true for any $x\in\partial B_{1}\left(\left(1-\gamma\right)e_{N}\right)\cap\overline{\R^{N}_{+}}$.
\begin{figure}[htbp]
\centering
\includegraphics[width=0.6\textwidth]{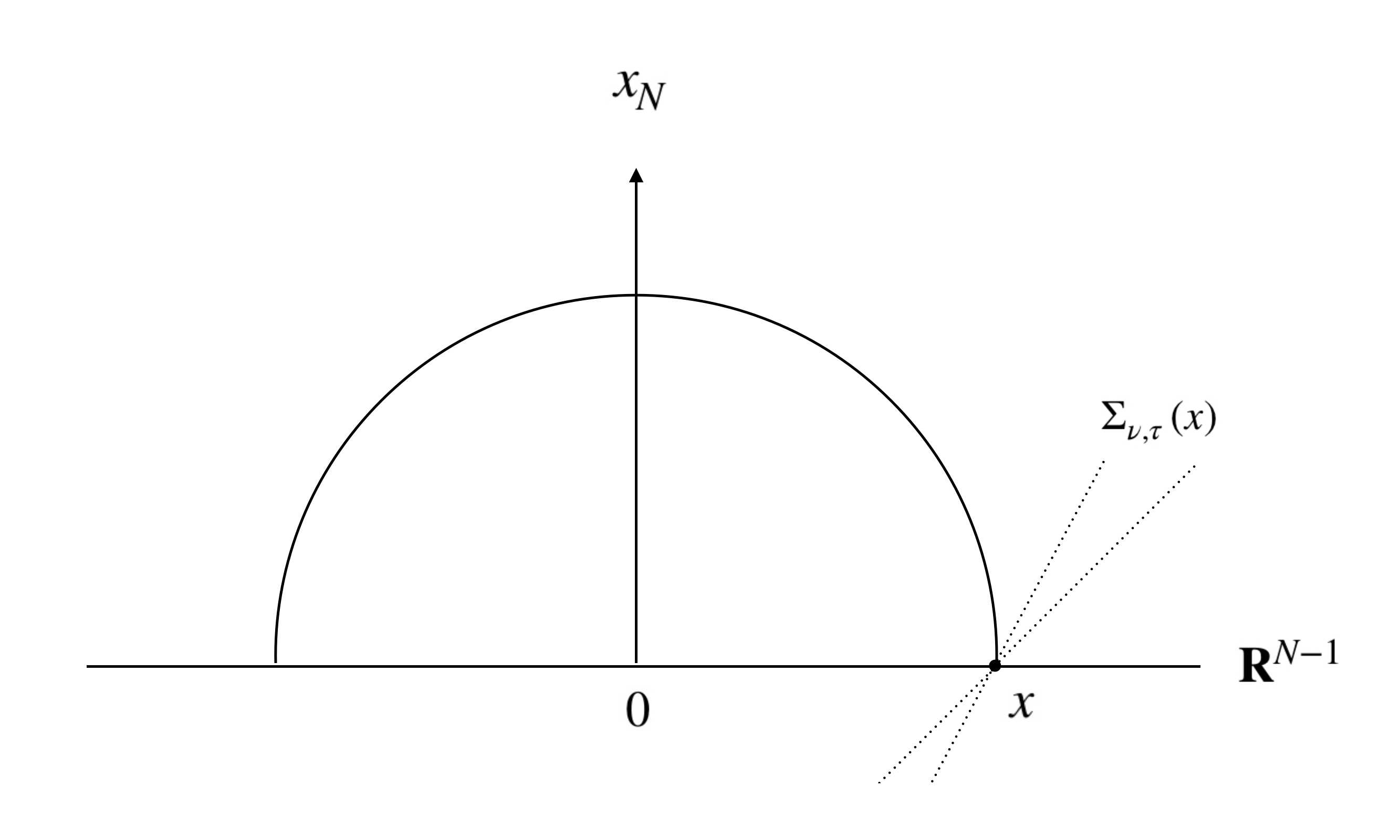}
\caption{$\left|\Sigma_{\nu,\tau}\left(x\right)\cap B_{1}\cap\R^{N}_{+}\right|=0$.}
\label{figure3}
\end{figure}
\begin{figure}[htbp]
\centering
\includegraphics[width=0.6\textwidth]{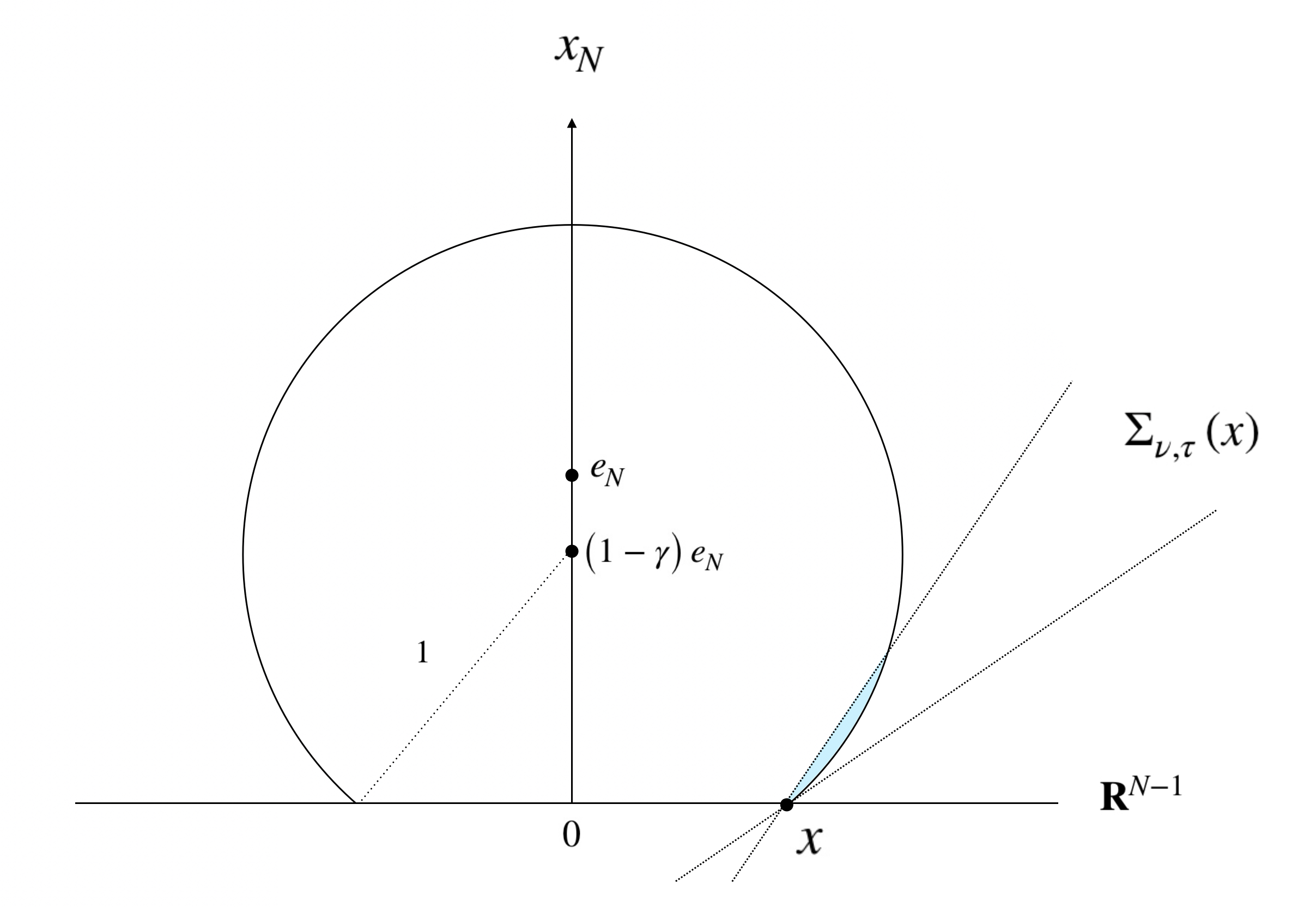}
\caption{The blue area $\left|\Sigma_{\nu,\tau}\left(x\right)\cap B_{1}\left(\left(1-\gamma\right)e_{N}\right)\cap\R^{N}_{+}\right|>0$.}
\label{figure4}
\end{figure}
\end{rem}

The above considerations are summarized in the following

\begin{cor}\label{3-5}
Given $\nu\in{\mathbb S}^{N-1}$ and $\tau\in(0,1]$, then there exist $\gamma=\gamma\left(\nu,\tau\right)\in\left(0,1\right)$ such that for any $x\in\partial B_{1}\left(\left(1-\gamma\right)e_{N}\right)\cap\overline{\R^{N}_{+}}$,
\begin{align*}
\left|\Sigma_{\nu,\tau}\left(x\right)\cap B_{1}\left(\left(1-\gamma\right)e_{N}\right)\cap\R^{N}_{+}\right|>0.
\end{align*}
\end{cor}

Next theorem is the main result from which Theorem \ref{th1} easily follows, as proved at the end of this section.

\begin{thm}\label{3-3}
Assume $1\leq p\leq\frac{N+s}{N-s}$. If $u\in C^{2}\left(\R_{+}^{N}\right)\cap\mathcal{L}_{s}$ is a nonnegative solution of \eqref{2-5-1}   satisfying \eqref{2-0-1}, then $u\equiv0$.
\end{thm}
\begin{proof}
For the $\nu_{0},\tau_{0}$ given in \eqref{1-0-2}, by Corollary \ref{3-5}, there exists $0<\gamma_{0}<1$ such that for any $x\in\partial B_{1}\left(\left(1-\gamma_{0}\right)e_{N}\right)\cap\overline{\R^{N}_{+}}$,
\begin{align}\label{3-3-1}
\left|\Sigma_{\nu_{0},\tau_{0}}\left(x\right)\cap B_{1}\left(\left(1-\gamma_{0}\right)e_{N}\right)\cap\R^{N}_{+}\right|>0.
\end{align}
We choose $\varphi\in C^{\infty}_{0}\left(\R^{N}\right)$ such that  $0<\varphi\leq1$ in $B_{1}\left(\left(1-\gamma_{0}\right)e_{N}\right)$ and
\begin{align*}
\varphi\left(x\right)=
\left\lbrace 
\begin{aligned}
&1,&&x\in B_{1-\frac{\gamma_{0}}{2}}\left(\left(1-\gamma_{0}\right)e_{N}\right),\\
&0,&&x\notin B_{1}\left(\left(1-\gamma_{0}\right)e_{N}\right),
\end{aligned}
\right.
\end{align*}
see Figure \ref{figure5}. 
\begin{figure}[htbp]
\centering
\includegraphics[width=0.7\textwidth]{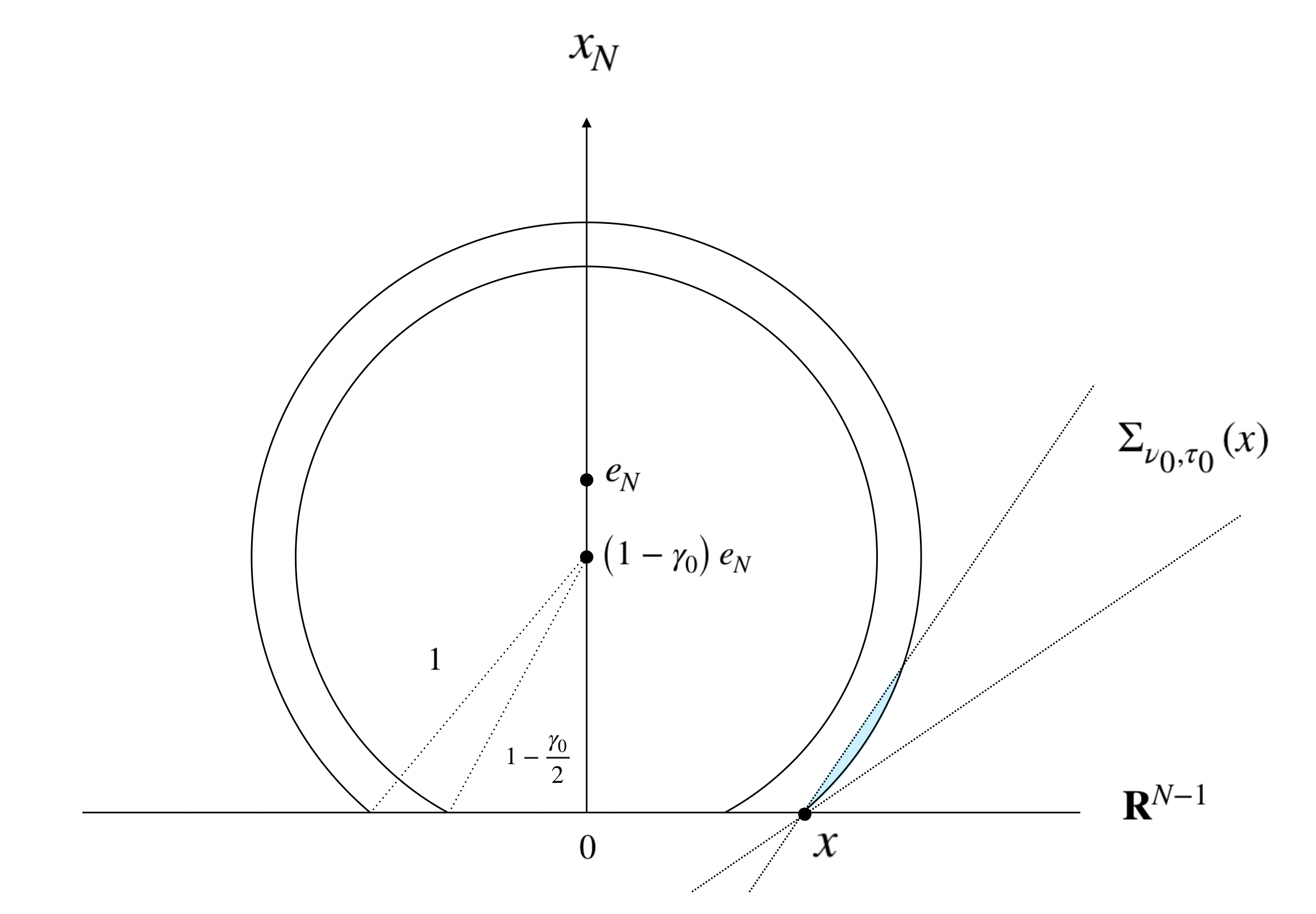}
\caption{The graph of $\varphi$.}
\label{figure5}
\end{figure}

For any $s\leq\alpha<\min\left\{1,2s\right\}$, we define
\begin{align*}
\phi_{\alpha}\left(x\right):=w_\alpha(x)\varphi\left(x\right):=\left(x_{N}\right)_{+}^{\alpha}\varphi\left(x\right).
\end{align*}

\noindent
\textbf{Step 1.} For any given $\alpha_{0}\in\left(s,\min\left\{1,2s\right\}\right)$, we show that there exists $M>0$ such that
\begin{align}\label{3-3-2}
-L\phi_{\alpha_{0}}\left(x\right)-L\phi_{s}\left(x\right)\leq M\phi_{s}\left(x\right)\quad\forall x\in\R_{+}^{N}.
\end{align}
For this we first note that for any $x\in\R_{+}^{N}\cap\left\{\left|x-\left(1-\gamma_{0}\right)e_{N}\right|\geq1\right\}$, by the assumption \eqref{1-0-1} we infer that 
\begin{align*}
-L\phi_{\alpha_{0}}\left(x\right)-L\phi_{s}\left(x\right)
=&-\int_{\R^{N}}\frac{\phi_{\alpha_{0}}\left(x+y\right)+\phi_{\alpha_{0}}\left(x-y\right)}{\left|y\right|^{N+2s}}a\left(\frac{y}{\left|y\right|}\right)dy\\
&-\int_{\R^{N}}\frac{\phi_{s}\left(x+y\right)+\phi_{s}\left(x-y\right)}{\left|y\right|^{N+2s}}a\left(\frac{y}{\left|y\right|}\right)dy\leq0=M\phi_{s}\left(x\right).
\end{align*}
Hence we only need to show that 
\begin{align}\label{3-3-3}
\inf_{x\in\R_{+}^{N}\cap B_{1}\left(\left(1-\gamma_{0}\right)e_{N}\right)}\frac{L\phi_{\alpha_{0}}\left(x\right)+L\phi_{s}\left(x\right)}{\phi_{s}\left(x\right)}>-\infty.
\end{align}
Assuming that \eqref{3-3-3} is not true, then there exists a convergent sequence $\left\{x_{n}\right\}_n\subset\R_{+}^{N}\cap B_{1}\left(\left(1-\gamma_{0}\right)e_{N}\right)$ such that
\begin{align}\label{eq9}
\lim_{n\rightarrow+\infty}\frac{L\phi_{\alpha_{0}}\left(x_{n}\right)+L\phi_{s}\left(x_{n}\right)}{\phi_{s}\left(x_{n}\right)}=-\infty.
\end{align}
Let $x_{n}\rightarrow x_{\infty}\in\overline{\R_{+}^{N}\cap B_{1}\left(\left(1-\gamma_{0}\right)e_{N}\right)}$ as $n\rightarrow+\infty$. We distinguish different cases,  each of them producing a contradiction to \eqref{eq9}.

\bigskip
\noindent Case 1: $x_{\infty}\in\R_{+}^{N}\cap B_{1}\left(\left(1-\gamma_{0}\right)e_{N}\right)$. We obtain 
\begin{align*}
\lim_{n\rightarrow+\infty} L\phi_{\alpha_{0}}\left(x_{n}\right)+L\phi_{s}\left(x_{n}\right)=L\phi_{\alpha_{0}}\left(x_{\infty}\right)+L\phi_{s}\left(x_{\infty}\right)
\end{align*}
and 
\begin{align*}
0<\lim_{n\rightarrow+\infty} \phi_{s}\left(x_{n}\right)=\phi_{s}\left(x_{\infty}\right)<+\infty,
\end{align*}
then
\begin{align*}
\lim_{n\rightarrow+\infty}\frac{L\phi_{\alpha_{0}}\left(x_{n}\right)+L\phi_{s}\left(x_{n}\right)}{\phi_{s}\left(x_{n}\right)}=\frac{L\phi_{\alpha_{0}}\left(x_{\infty}\right)+L\phi_{s}\left(x_{\infty}\right)}{\phi_{s}\left(x_{\infty}\right)}.
\end{align*}

\medskip
\noindent
Case 2: $x_{\infty}\in\R_{+}^{N}\cap \partial B_{1}\left(\left(1-\gamma_{0}\right)e_{N}\right)$. Recalling that $a\geq0$,  by \eqref{1-0-2} and \eqref{3-3-1} we have
\begin{align*}
\lim_{n\rightarrow+\infty}L\phi_{\alpha_{0}}\left(x_{n}\right)+L\phi_{s}\left(x_{n}\right)
&=\int_{\R^{N}}\frac{\phi_{\alpha_{0}}\left(x_{\infty}+y\right)+\phi_{\alpha_{0}}\left(x_{\infty}-y\right)}{\left|y\right|^{N+2s}}a\left(\frac{y}{\left|y\right|}\right)dy\\
&\quad+\int_{\R^{N}}\frac{\phi_{s}\left(x_{\infty}+y\right)+\phi_{s}\left(x_{\infty}-y\right)}{\left|y\right|^{N+2s}}a\left(\frac{y}{\left|y\right|}\right)dy\\
&\geq d\int_{\Sigma_{\nu_{0},\tau_{0}}\left(0\right)}\frac{\phi_{\alpha_{0}}\left(x_{\infty}+y\right)+\phi_{\alpha_{0}}\left(x_{\infty}-y\right)}{\left|y\right|^{N+2s}}dy\\
&\quad+d\int_{\Sigma_{\nu_{0},\tau_{0}}\left(0\right)}\frac{\phi_{s}\left(x_{\infty}+y\right)+\phi_{s}\left(x_{\infty}-y\right)}{\left|y\right|^{N+2s}}dy>0.
\end{align*}
Moreover, $\phi_{s}\left(x_{n}\right)>0$ for any $n$ and
\begin{align*}
\lim_{n\rightarrow+\infty} \phi_{s}\left(x_{n}\right)=0,
\end{align*}
thus
\begin{align*}
\lim_{n\rightarrow+\infty}\frac{L\phi_{\alpha_{0}}\left(x_{n}\right)+L\phi_{s}\left(x_{n}\right)}{\phi_{s}\left(x_{n}\right)}=+\infty.
\end{align*}

\medskip
\noindent
Case 3: $x_{\infty}\in\partial\R_{+}^{N}\cap B_{1}\left(\left(1-\gamma_{0}\right)e_{N}\right)$. By Lemmas \ref{2-3} and \ref{2-1}-(2), for any $x\in \R^{N}_{+}$
\begin{equation}\label{3-3-8}
\begin{split}
L\phi_{\alpha_{0}}\left(x\right)+L\phi_{s}\left(x\right)
&=C_{\alpha_{0}}\varphi\left(x\right)x^{\alpha_{0}-2s}_{N}+x_{N}^{\alpha_{0}}L\varphi\left(x\right)+x_{N}^{s}L\varphi\left(x\right)\\
&\quad+l\left[w_{\alpha_{0}},\varphi\right]\left(x\right)+l\left[w_{s},\varphi\right]\left(x\right),
\end{split}
\end{equation}
where $C_{\alpha_{0}}>0$ and 
\begin{align*}
&l\left[w_{\alpha_{0}},\varphi\right]\left(x\right)+l\left[w_{s},\varphi\right]\left(x\right)\\
&=\int_{\R^{N}}\left\{\frac{\left[\varphi\left(x+y\right)-\varphi\left(x\right)\right]\left[\left(x_{N}+y_{N}\right)_{+}^{\alpha_{0}}+\left(x_{N}+y_{N}\right)_{+}^{s}-x_{N}^{\alpha_{0}}-x_{N}^{s}\right]}{\left|y\right|^{N+2s}}\right.\\
&\quad\left.+\frac{\left[\varphi\left(x-y\right)-\varphi\left(x\right)\right]\left[\left(x_{N}-y_{N}\right)_{+}^{\alpha_{0}}+\left(x_{N}-y_{N}\right)_{+}^{s}-x_{N}^{\alpha_{0}}-x_{N}^{s}\right]}{\left|y\right|^{N+2s}}\right\}a\left(\frac{y}{\left|y\right|}\right)dy.
\end{align*}
We see that
\begin{align*}
\lim_{n\rightarrow+\infty}C_{\alpha_{0}}\varphi\left(x_{n}\right)\cdot\left(x_{n}\right)_{N}^{\alpha_{0}-2s}=+\infty
\end{align*}
and
\begin{align*}
\lim_{n\rightarrow+\infty}\left(x_{n}\right)_{N}^{\alpha_{0}}L\varphi\left(x_{n}\right)+\left(x_{n}\right)_{N}^{s}L\varphi\left(x_{n}\right)=0.
\end{align*}
Moreover Lemma \ref{3-2} yields
\begin{align*}
\lim_{n\rightarrow+\infty}l\left[w_{\alpha_{0}},\varphi\right]\left(x_{n}\right)+l\left[w_{s},\varphi\right]\left(x_{n}\right)=l\left[w_{\alpha_{0}},\varphi\right]\left(x_{\infty}\right)+l\left[w_{s},\varphi\right]\left(x_{\infty}\right),
\end{align*}
hence by \eqref{3-3-8},
\begin{align*}
\lim_{n\rightarrow+\infty}L\phi_{\alpha_{0}}\left(x_{n}\right)+L\phi_{s}\left(x_{n}\right)=+\infty.
\end{align*}
By the fact
\begin{align*}
\lim_{n\rightarrow+\infty} \phi_{s}\left(x_{n}\right)=0,
\end{align*}
and $\phi_{s}\left(x_{n}\right)>0$ for any $n$, we get
\begin{align*}
\lim_{n\rightarrow+\infty}\frac{L\phi_{\alpha_{0}}\left(x_{n}\right)+L\phi_{s}\left(x_{n}\right)}{\phi_{s}\left(x_{n}\right)}=+\infty.
\end{align*}

\medskip
\noindent
Case 4: $x_{\infty}\in\partial\R_{+}^{N}\cap \partial B_{1}\left(\left(1-\gamma_{0}\right)e_{N}\right)$. Notice that by \eqref{3-3-8}, since $\varphi(x_n)>0$ for any $n$,
\begin{equation}\label{3-3-10}
\begin{split}
L\phi_{\alpha_{0}}\left(x_{n}\right)+L\phi_{s}\left(x_{n}\right)
&\geq\left(x_{n}\right)_{N}^{\alpha_{0}}L\varphi\left(x_{n}\right)+\left(x_{n}\right)_{N}^{s}L\varphi\left(x_{n}\right)\\
&\quad+l\left[w_{\alpha_{0}},\varphi\right]\left(x_{n}\right)+l\left[w_{s},\varphi\right]\left(x_{n}\right).
\end{split}
\end{equation}
It is obvious that
\begin{align}\label{3-3-11}
\lim_{n\rightarrow+\infty}\left(x_{n}\right)_{N}^{\alpha_{0}}L\varphi\left(x_{n}\right)+\left(x_{n}\right)_{N}^{s}L\varphi\left(x_{n}\right)=0.
\end{align}
Since $a\geq0$, by \eqref{1-0-2}, Lemma \ref{3-2} and \eqref{3-3-1}  we have
\begin{align}\label{3-3-12}
\nonumber\lim_{n\rightarrow+\infty}l\left[w_{\alpha_{0}},\varphi\right]\left(x_{n}\right)+l\left[w_{s},\varphi\right]\left(x_{n}\right)
&=\int_{\R^{N}}\frac{\phi_{\alpha_{0}}\left(x_{\infty}+y\right)+\phi_{\alpha_{0}}\left(x_{\infty}-y\right)}{\left|y\right|^{N+2s}}a\left(\frac{y}{\left|y\right|}\right)dy\\
\nonumber&\quad+\int_{\R^{N}}\frac{\phi_{s}\left(x_{\infty}+y\right)+\phi_{s}\left(x_{\infty}-y\right)}{\left|y\right|^{N+2s}}a\left(\frac{y}{\left|y\right|}\right)dy\\
\nonumber&\geq d\int_{\Sigma_{\nu_{0},\tau_{0}}\left(0\right)}\frac{\phi_{\alpha_{0}}\left(x_{\infty}+y\right)+\phi_{\alpha_{0}}\left(x_{\infty}-y\right)}{\left|y\right|^{N+2s}}dy\\
&\quad+d\int_{\Sigma_{\nu_{0},\tau_{0}}\left(0\right)}\frac{\phi_{s}\left(x_{\infty}+y\right)+\phi_{s}\left(x_{\infty}-y\right)}{\left|y\right|^{N+2s}}dy>0.
\end{align}
Combining \eqref{3-3-10}, \eqref{3-3-11} and \eqref{3-3-12} with the fact
\begin{align*}
\lim_{n\rightarrow+\infty} \phi_{s}\left(x_{n}\right)=0
\end{align*}
and $\phi_{s}\left(x_{n}\right)>0$ for any $n$,
even in this case we get :
\begin{align*}
\lim_{n\rightarrow+\infty}\frac{L\phi_{\alpha_{0}}\left(x_{n}\right)+L\phi_{s}\left(x_{n}\right)}{\phi_{s}\left(x_{n}\right)}=+\infty.
\end{align*}
In conclusion, we have reached a contradiction to \eqref{eq9}  for any $x_{\infty}\in\overline{\R_{+}^{N}\cap B_{1}\left(\left(1-\gamma_{0}\right)e_{N}\right)}$, then inequality \eqref{3-3-2} holds.

\bigskip
\noindent
\textbf{Step 2.} For any $R>0$, we make the following rescaling
\begin{align*}
\varphi_{R}\left(x\right)=\varphi\left(\frac{x}{R}\right)=
\left\lbrace 
\begin{aligned}
&1,&&x\in B_{R\left(1-\frac{\gamma_{0}}{2}\right)}\left(R\left(1-\gamma_{0}\right)e_{N}\right),\\
&0,&&x\notin B_{R}\left(R\left(1-\gamma_{0}\right)e_{N}\right).
\end{aligned}
\right.
\end{align*}
For any $s\leq\alpha<\min\left\{1,2s\right\}$, we define
\begin{align*}
\phi_{\alpha,R}\left(x\right)=\left(x_{N}\right)_{+}^{\alpha}\varphi_{R}\left(x\right).
\end{align*}
Since $u$ is solution of \eqref{2-5-1}, by Proposition \ref{2-4} and Lemma \ref{2-1}-(1) we obtain
\begin{align*}
\int_{\R_{+}^{N}}u^{p}\phi_{s,R}dx
&\leq-R^{s}\int_{\R^{N}}\left(\frac{x_{N}}{R}\right)_{+}^{s}\varphi_{R}Ludx\\
&\leq-R^{s}\int_{\R^{N}}\left(\frac{x_{N}}{R}\right)_{+}^{s}\varphi_{R}Ludx-R^{s}\int_{\R^{N}}\left(\frac{x_{N}}{R}\right)_{+}^{\alpha_{0}}\varphi_{R}Ludx\\
&=-R^{-s}\int_{\R^{N}}uL\phi_{s}\left(\frac{x}{R}\right)dx-R^{-s}\int_{\R^{N}}uL\phi_{\alpha_{0}}\left(\frac{x}{R}\right)dx\\
&=R^{-s}\int_{\R_{+}^{N}}u\left[-L\phi_{\alpha_{0}}\left(\frac{x}{R}\right)-L\phi_{s}\left(\frac{x}{R}\right)\right]dx.
\end{align*}
Now we use \eqref{3-3-2} to infer that
\begin{align*}
\int_{\R_{+}^{N}}u\left[-L\phi_{\alpha_{0}}\left(\frac{x}{R}\right)-L\phi_{s}\left(\frac{x}{R}\right)\right]dx
\leq M\int_{\R_{+}^{N}}u\phi_{s}\left(\frac{x}{R}\right)dx=MR^{-s}\int_{\R_{+}^{N}}u\phi_{s,R}dx
\end{align*}
and that
\begin{align}\label{3-3-4}
\int_{\R_{+}^{N}}u^{p}\phi_{s,R}dx\leq MR^{-2s}\int_{\R_{+}^{N}}u\phi_{s,R}dx.
\end{align}
Since $0\leq\phi_{s,R}\leq(x_N)_+^s$  and $B_{R}\left(R\left(1-\gamma_{0}\right)e_{N}\right)\subset B_{2R}$,  we may apply the H\"{o}lder inequality to get, for any $p>1$,
\begin{align*}
\int_{\R_{+}^{N}}u\phi_{s,R}dx
&\leq\left(\int_{\R_{+}^{N}\cap B_{R}\left(R\left(1-\gamma_{0}\right)e_{N}\right)}u^{p}\phi_{s,R}dx\right)^{\frac{1}{p}}\left(\int_{\R_{+}^{N}\cap B_{R}\left(R\left(1-\gamma_{0}\right)e_{N}\right)}\phi_{s,R}dx\right)^{\frac{p-1}{p}}\\
&\leq C\left(N,p\right)R^{\frac{\left(N+s\right)\left(p-1\right)}{p}}\left(\int_{\R_{+}^{N}}u^{p}\phi_{s,R}dx\right)^{\frac{1}{p}}.
\end{align*}
Therefore, by \eqref{3-3-4},
\begin{align}\label{3-3-5}
\int_{\R_{+}^{N}}u^{p}\phi_{s,R}dx\leq C\left(N,M,p\right)R^{N+s-\frac{2sp}{p-1}}.
\end{align}

\bigskip
\noindent
\textbf{Step 3.} In the case $p=1$, we immediately obtain $u\equiv0$ by letting $R\rightarrow+\infty$ in \eqref{3-3-4}.\\ We now consider  $p>1$. Since  $\varphi_{R}\rightarrow\varphi_{\infty}\equiv1$ in $\R^{N}_{+}$ and $\phi_{s,R}\rightarrow\phi_{s,\infty}=x_{N}^{s}$ as $R\rightarrow+\infty$, if we let $R\rightarrow+\infty$ in \eqref{3-3-5}, we have $u\equiv0$ provided $p<\frac{N+s}{N-s}$. \\ If instead $p=\frac{N+s}{N-s}$, then we infer that
\begin{align}\label{3-3-6}
\int_{\R_{+}^{N}}x^{s}_{N}u^{p}dx<+\infty.
\end{align}
We can rewrite
\begin{align*}
\int_{\R_{+}^{N}}u\phi_{s,R}dx
=\int_{\R_{+}^{N}\cap B_{\sqrt{R}}}u\phi_{s,R}dx+\int_{\R_{+}^{N}\cap\left\{\sqrt{R}\leq\left|x\right|\leq 2R\right\}}u\phi_{s,R}dx.
\end{align*}
Using once again the fact $0\leq\phi_{s,R}\leq(x_N)_+^s$, by the H\"{o}lder inequality we obtain
\begin{align*}
\int_{\R_{+}^{N}\cap B_{\sqrt{R}}}u\phi_{s,R}dx
&\leq\left(\int_{\R_{+}^{N}\cap B_{\sqrt{R}}}x_{N}^{s}u^{p}dx\right)^{\frac{1}{p}}\left(\int_{\R_{+}^{N}\cap B_{\sqrt{R}}}x_{N}^{s}dx\right)^{\frac{p-1}{p}}\\
&\leq C\left(N,p\right)R^{\frac{\left(N+s\right)\left(p-1\right)}{2p}}\left(\int_{\R_{+}^{N}}x^{s}_{N}u^{p}dx\right)^{\frac{1}{p}}\\
&=C\left(N,p\right)R^{s}\left(\int_{\R_{+}^{N}}x^{s}_{N}u^{p}dx\right)^{\frac{1}{p}}
\end{align*}
and
\begin{align*}
\int_{\R_{+}^{N}\cap\left\{\sqrt{R}\leq\left|x\right|\leq 2R\right\}}u\phi_{s,R}dx
&\leq\left(\int_{\R_{+}^{N}\cap\left\{\sqrt{R}\leq\left|x\right|\leq 2R\right\}}x_{N}^{s}u^{p}dx\right)^{\frac{1}{p}}\left(\int_{\R_{+}^{N}\cap\left\{\sqrt{R}\leq\left|x\right|\leq 2R\right\}}x_{N}^{s}dx\right)^{\frac{p-1}{p}}\\
&\leq C\left(N,p\right)R^{\frac{\left(N+s\right)\left(p-1\right)}{p}}\left(\int_{\R_{+}^{N}\cap\left\{\sqrt{R}\leq\left|x\right|\leq 2R\right\}}x^{s}_{N}u^{p}dx\right)^{\frac{1}{p}}\\
&=C\left(N,p\right)R^{2s}\left(\int_{\R_{+}^{N}\cap\left\{\sqrt{R}\leq\left|x\right|\leq 2R\right\}}x^{s}_{N}u^{p}dx\right)^{\frac{1}{p}}.
\end{align*}
By \eqref{3-3-4}, we get
\begin{align*}
\int_{\R_{+}^{N}}u^{p}\phi_{s,R}dx\leq C\left(N,M,p\right)\left[R^{-s}\left(\int_{\R_{+}^{N}}x^{s}_{N}u^{p}dx\right)^{\frac{1}{p}}+\left(\int_{\R_{+}^{N}\cap\left\{\sqrt{R}\leq\left|x\right|\leq 2R\right\}}x^{s}_{N}u^{p}dx\right)^{\frac{1}{p}}\right].
\end{align*}
Therefore, by \eqref{3-3-6} the right-hand side of the above inequality goes to $0$ as  $R\rightarrow+\infty$ and we again conclude that $u\equiv0$.
\end{proof}

Now we are ready to prove Theorem \ref{th1}.

\begin{proof}[Proof of Theorem \ref{th1}]
Let $u$ be any classical solution of \eqref{1-0-3}. By Lemma \ref{2-5}, the function $\widetilde{u}$, defined in \eqref{eq10}, is a nonnegative classical solution of \eqref{2-5-1}. Theorem \ref{3-3} yields $\widetilde{u}\equiv0$ in $\R^N$, hence using the continuity of $u$ in $\R^N_+$ we infer that
\begin{align*}
u
\left\lbrace 
\begin{aligned}
&\geq0,&&x_{N}<1,\\
&=0,&&x_{N}\geq1.
\end{aligned}
\right.
\end{align*}
For any $\overline{x}$ with $\overline{x}_{N}=1$, we have
\begin{align*}
0=u^{p}\left(\overline{x}\right)\leq-Lu\left(\overline{x}\right)=-\int_{\R^{N}}\frac{u\left(\overline{x}+y\right)+u\left(\overline{x}-y\right)}{\left|y\right|^{N+2s}}a\left(\frac{y}{\left|y\right|}\right)dy\leq0
\end{align*}
and $Lu\left(\overline{x}\right)=0$. Since $u\in C^{2}\left(\R_{+}^{N}\right)$, then $u\equiv0$ in $\Sigma_{\nu_{0},\tau_{0}}\left(\overline{x}\right)\cap\R^{N}_{+}$. The arbitrariness of $\overline{x}$ implies $u\equiv0$ in $\R^{N}_{+}$ and moreover $u=0$ a.e. in $\overline{\R_{-}^{N}}$, see Figure \ref{figure6}. 
\begin{figure}[htbp]
\centering
\includegraphics[width=0.6\textwidth]{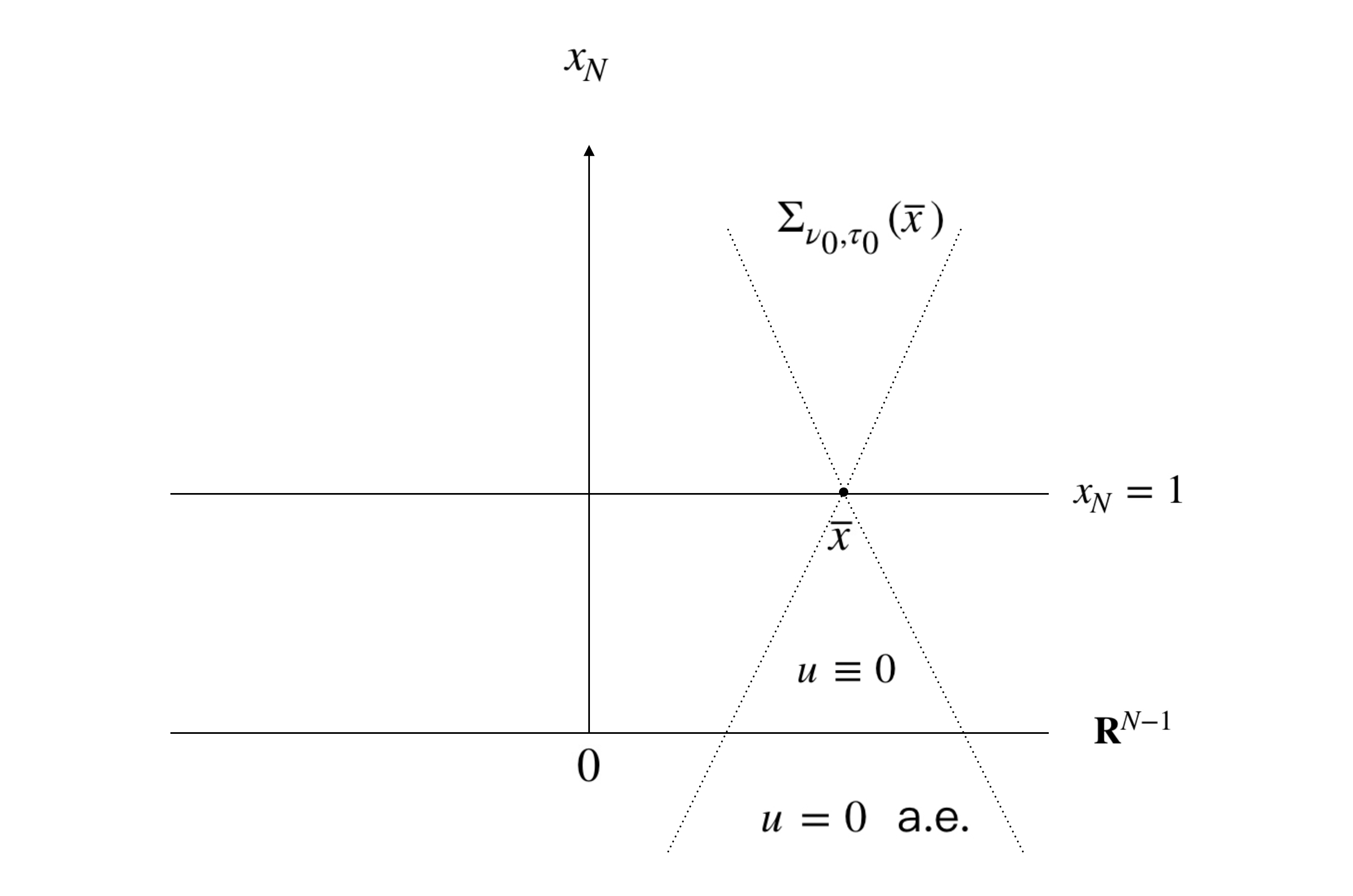}
\caption{}
\label{figure6}
\end{figure}

\end{proof}

\section{Optimality of the critical exponent \texorpdfstring{$\frac{N+s}{N-s}$}{Lg}}\label{s4}

We prove that the exponent $p=\frac{N+s}{N-s}$ in Theorem \ref{th1} is sharp, in the sense that there is at least some operator in the class \eqref{eq1} for which  problem \eqref{1-0-3} has nontrivial solutions as soon as $p>\frac{N+s}{N-s}$. In fact, this occurs whenever $a(\theta)$ is any positive constant function, in which case $-L$ coincides, up to a multiplicative constant, with $(-\Delta)^s$.

\begin{thm}\label{4-1}
For  any $p>\frac{N+s}{N-s}$ and $N\geq2$, the problem 
\begin{align}\label{4-1-1}
\left\lbrace 
\begin{aligned}
\left(-\Delta\right)^{s}u&\geq u^{p},&&x\in\R_{+}^{N},\\
u&=0,&&x\in\R_{-}^{N}
\end{aligned}
\right.
\end{align}
 admits  positive classical solutions.
\end{thm}
\begin{proof}
Let 
\begin{align*}
w_{\alpha}\left(x\right)=\left(x_{N}\right)^{\alpha}_{+},\quad0<\alpha<s.
\end{align*}
Consider the Kelvin transform of $w_{\alpha}$, given by
\begin{align*}
\overline{w}_{\alpha}\left(x\right):=\frac{1}{\left|x\right|^{N-2s}}w_{\alpha}\left(\frac{x}{\left|x\right|^{2}}\right)=\frac{\left(x_{N}\right)^{\alpha}_{+}}{\left|x\right|^{N-2s+2\alpha}}\in C^{2}\left(\R_{+}^{N}\right)\cap\mathcal{L}_{s}.
\end{align*}
By the general property of the Kelvin transformation
$$
\left(-\Delta\right)^{s}\overline{w}_{\alpha}(x)=\frac{1}{\left|x\right|^{N+2s}}\left(-\Delta\right)^{s}w_{\alpha}\left(\frac{x}{\left|x\right|^{2}}\right),
$$
see e.g. \cite[Proposition A.1]{ROS}, from Lemma \ref{2-3} we infer that for any $x\in\R^{N}_{+}$,
$$
-\left(-\Delta\right)^{s}\overline{w}_{\alpha}\left(x\right)=C_{\alpha}\frac{1}{\left|x\right|^{N+2s}}\left(\frac{x_{N}}{\left|x\right|^{2}}\right)^{\alpha-2s}=C_{\alpha}\frac{x_{N}^{\alpha-2s}}{\left|x\right|^{N-2s+2\alpha}},
$$ 
where $C_{\alpha}<0$.

For $0<\varepsilon\leq\left(-C_{\alpha}\right)^{\frac{1}{p-1}}$, we define the function
\begin{align*}
u\left(x\right)=\varepsilon\overline{w}_{\alpha}\left(x\right).
\end{align*}
For any $x\in\R_{+}^{N}$, then we have
\begin{align}\label{4-1-2}
-\left(-\Delta\right)^{s}u\left(x\right)+u^{p}\left(x\right)
\nonumber&=\varepsilon C_{\alpha}\frac{x_{N}^{\alpha-2s}}{\left|x\right|^{N-2s+2\alpha}}+\varepsilon^{p}\frac{x_{N}^{\alpha p}}{\left|x\right|^{\left(N-2s+2\alpha\right)p}}\\
&=\varepsilon\frac{x_{N}^{\alpha-2s}}{\left|x\right|^{N-2s+2\alpha}}\left(C_{\alpha}+\varepsilon^{p-1}\frac{x_{N}^{\alpha p-\alpha+2s}}{\left|x\right|^{\left(N-2s+2\alpha\right)\left(p-1\right)}}\right).
\end{align}

When $\frac{N+s}{N-s}<p<\frac{N}{N-2s}$, we can choose $\alpha\in\left(0,s\right)$ in such a way  
\begin{align*}
\alpha p-\alpha+2s=\left(N-2s+2\alpha\right)\left(p-1\right),
\end{align*}
i.e. $\alpha=\frac{N-\left(N-2s\right)p}{p-1}$. Then, from \eqref{4-1-2}, for any $x\in\R_{+}^{N}$ we have 
\begin{align}\label{4-1-3}
-\left(-\Delta\right)^{s}u\left(x\right)+u^{p}\left(x\right)
\leq\varepsilon\frac{x_{N}^{\alpha-2s}}{\left|x\right|^{N-2s+2\alpha}}\left(C_{\alpha}+\varepsilon^{p-1}\right)\leq0.
\end{align}
Thus $u$ is a classical solution of \eqref{4-1-1} when $\frac{N+s}{N-s}<p<\frac{N}{N-2s}$. 

If $p\geq\frac{N}{N-2s}$, then for any $\alpha\in\left(0,s\right)$ one has
\begin{align*}
\alpha p-\alpha+2s<\left(N-2s+2\alpha\right)\left(p-1\right).
\end{align*}
In view \eqref {4-1-2}, $u$ satisfies \eqref{4-1-3} for any $x_{N}\geq1$. As a consequence, the function
\begin{align*}
\widetilde{u}\left(x\right)=
\left\lbrace 
\begin{aligned}
&u\left(x+e_{N}\right),&&x\in\R_{+}^{N},\\
&0,&&x\in\overline{\R_{-}^{N}} 
\end{aligned}
\right.
\end{align*}
is in turn a solution of \eqref{4-1-1}.
\end{proof}

\begin{rem}\label{N12}
The existence result of Theorem \ref{4-1} remains valid even if $N=1$. We briefly discuss this aspect by distinguishing the cases $0<s<\frac12$ and  $\frac12\leq s<1$.
\newline
If $0<s<\frac12$, using the fact that the function $|x|^{-(1-2s)}$  is $s$-harmonic  in $\R\backslash\left\{0\right\}$, one can argue as in the proof of Theorem \ref{4-1} with $N$ replaced by $1$.
\newline
If instead $\frac12\leq s<1$, the function $|x|^{2s-1}$ is still $s$-harmonic  in $\R\backslash\left\{0\right\}$ (if $s=\frac12$ we simply mean the constant $1$) but unbounded at infinity if $s>\frac12$. The equivalent formulation  of \eqref{4-1-2}  for the function $$u(x)=\varepsilon\frac{(x)_+^\alpha}{|x|^{2\alpha-2s+1}}$$ is 
\begin{align*}
-\left(-\Delta\right)^{s}u\left(x\right)+u^{p}\left(x\right)
\nonumber&=\frac{\varepsilon C_{\alpha}}{x^{1+\alpha}}+\frac{\varepsilon^{p}}{x^{\left(1+\alpha-2s\right)p}}\\
&=\frac{\varepsilon}{x^{1+\alpha}}\left(C_{\alpha}+\frac{\varepsilon^{p-1}}{x^{\left(1+\alpha-2s\right)p-1-\alpha}}\right)\qquad\forall x\in\R_+,
\end{align*}
where $C_\alpha<0$. In this case we can choose $\alpha=\frac{1+(2s-1)p}{p-1}$ is such a way
$$
(1+\alpha-2s)p=\alpha+1
$$
and, differently from the previous cases $N=1$ and $0<s<\frac12$ or $N\ge2$, the exponent $\alpha$ is now always positive since, for any $p>\frac{1+s}{1-s}$, it turns out that $\frac{1+\left(2s-1\right)p}{p-1}\in\left(2s-1,s\right)$. Hence
$$
-\left(-\Delta\right)^{s}u\left(x\right)+u^{p}\left(x\right)\leq0\quad\forall x\in\R_+.
$$
\end{rem}

\section{Nonexistence in the whole space \texorpdfstring{$\R^{N}$}{Lg}}\label{s5}
We prove in this section that below the nonlocal Serrin exponent, the only entire non negative 
supersolution in $\R^N$ is the trivial one. The proof is similar to the case of the half space but simpler and 
it is given for completeness and clarity sake.
\begin{thm}\label{5-1}
Let $N\geq2$ and $1\leq p\leq\frac{N}{N-2s}$. If $u\in C^{2}\left(\R^{N}\right)\cap\mathcal{L}_{s}$ is a nonnegative solution of 
\begin{align}\label{eq12}
-Lu\geq u^{p}\quad \text{in $\R^N$},
\end{align}
then $u\equiv0$.
\end{thm}
\begin{proof}
We choose $\varphi\in C^{\infty}_{0}\left(\R^{N}\right)$ such that $0<\varphi\leq1$ in $B_{2}$ and
\begin{align*}
\varphi\left(x\right)=
\left\lbrace 
\begin{aligned}
&1,&&x\in B_{1},\\
&0,&&x\notin B_{2}.
\end{aligned}
\right.
\end{align*}
We claim that there exists $M>0$ such that
\begin{align}\label{5-1-1}
-L\varphi\left(x\right)\leq M\varphi\left(x\right)\quad\forall x\in\R^{N}.
\end{align}
By the assumption \eqref{1-0-1},  for any $\left|x\right|\geq2$ 
\begin{align*}
-L\varphi\left(x\right)=-\int_{\R^{N}}\frac{\varphi\left(x+y\right)+\varphi\left(x-y\right)}{\left|y\right|^{N+2s}}a\left(\frac{y}{\left|y\right|}\right)dy\leq0=M\varphi\left(x\right).
\end{align*}
Hence \eqref{5-1-1} is equivalent to  
\begin{align}\label{5-1-2}
\inf_{x\in B_{2}}\frac{L\varphi\left(x\right)}{\varphi\left(x\right)}>-\infty.
\end{align}

Assume \eqref{5-1-2} does not hold, then there exists a convergent sequence $\left\{x_{n}\right\}_n\subset B_{2}$ such that
\begin{align}\label{5-1-3}
\lim_{n\rightarrow+\infty}\frac{L\varphi\left(x_{n}\right)}{\varphi\left(x_{n}\right)}=-\infty.
\end{align}
Let $x_{n}\rightarrow x_{\infty}\in\overline{B_{2}}$ as $n\rightarrow+\infty$. We distinguish two cases. 

\bigskip
\noindent Case 1: $\left|x_{\infty}\right|<2$. In this case
\begin{align*}
\lim_{n\rightarrow+\infty}\frac{L\varphi\left(x_{n}\right)}{\varphi\left(x_{n}\right)}=\frac{L\varphi\left(x_{\infty}\right)}{\varphi\left(x_{\infty}\right)}
\end{align*}
which is a finite quantity since $\varphi\in C^\infty_0\left(\R^N\right)$ and $\varphi\left(x_\infty\right)>0$.

\bigskip
\noindent Case 2: $\left|x_{\infty}\right|=2$. By the assumptions \eqref{1-0-1}-\eqref{1-0-2} and using Lemma \ref{3-1}, we have
\begin{align*}
\lim_{n\rightarrow+\infty}L\varphi\left(x_{n}\right)
&=\int_{\R^{N}}\frac{\varphi\left(x_{\infty}+y\right)+\varphi\left(x_{\infty}-y\right)}{\left|y\right|^{N+2s}}a\left(\frac{y}{\left|y\right|}\right)dy\\
&\geq d\int_{\Sigma_{\nu_{0},\tau_{0}}(0)}\frac{\varphi\left(x_{\infty}+y\right)+\varphi\left(x_{\infty}-y\right)}{\left|y\right|^{N+2s}}dy>0.
\end{align*}
Thus
\begin{align*}
\lim_{n\rightarrow+\infty}\frac{L\varphi\left(x_{n}\right)}{\varphi\left(x_{n}\right)}=+\infty.
\end{align*}
Therefore, the assumption \eqref{5-1-3} can not occur, so that \eqref{5-1-1} holds. 

\smallskip

For any $R>0$, we consider the rescaled test-function
\begin{align*}
\varphi_{R}\left(x\right)=\varphi\left(\frac{x}{R}\right).
\end{align*}
Multiplying \eqref{eq12} by $\varphi_{R}$, integrating by parts, then using Lemma \ref{2-1}-(1) and \eqref{5-1-1} we have
\begin{align}\label{5-1-4}
\int_{\R^{N}}u^{p}\varphi_{R}dx
\leq-\int_{\R^{N}}uL\varphi_{R}dx
\leq MR^{-2s}\int_{\R^{N}}u\varphi_{R}dx.
\end{align}
By the H\"{o}lder inequality,
\begin{align*}
\int_{\R^{N}}u\varphi_{R}dx
\nonumber&\leq\left(\int_{B_{2R}}u^{p}\varphi_{R}dx\right)^{\frac{1}{p}}\left(\int_{B_{2R}}\varphi_{R}dx\right)^{\frac{p-1}{p}}\\
&\leq C\left(N,p\right)R^{\frac{N\left(p-1\right)}{p}}\left(\int_{\R^{N}}u^{p}\varphi_{R}dx\right)^{\frac{1}{p}}.
\end{align*}
Thus, by \eqref{5-1-4}, we obtain that for any $p>1$
\begin{align}\label{5-1-5}
\int_{\R^{N}}u^{p}\varphi_{R}dx\leq C\left(N,M,p\right)R^{N-\frac{2sp}{p-1}}.
\end{align}

In the case $p=1$, we let $R\rightarrow+\infty$ in \eqref{5-1-4} to get $u\equiv0$. When $1<p<\frac{N}{N-2s}$, letting $R\rightarrow+\infty$ in \eqref{5-1-5}, we infer that $u\equiv0$. If $p=\frac{N}{N-2s}$, then \eqref{5-1-5} yields, in the limit as $R\rightarrow+\infty$, that
\begin{align}\label{5-1-6}
\int_{\R^{N}}u^{p}dx<+\infty.
\end{align}
Moreover, in this case, we can rewrite
\begin{align*}
\int_{\R^{N}}u\varphi_{R}dx
=\int_{\left|x\right|\leq\sqrt{R}}u\varphi_{R}dx+\int_{\sqrt{R}\leq\left|x\right|\leq 2R}u\varphi_{R}dx.
\end{align*}
By the H\"{o}lder inequality
\begin{align*}
\int_{\left|x\right|\leq\sqrt{R}}u\varphi_{R}dx
&\leq C\left(N,p\right)R^{s}\left(\int_{\R^{N}}u^{p}dx\right)^{\frac{1}{p}}
\end{align*}
and
\begin{align*}
\int_{\sqrt{R}\leq\left|x\right|\leq 2R}u\varphi_{R}dx
&\leq C\left(N,p\right)R^{2s}\left(\int_{\sqrt{R}\leq\left|x\right|\leq 2R}u^{p}dx\right)^{\frac{1}{p}}.
\end{align*}
Therefore, by \eqref{5-1-4}, it follows that
\begin{align*}
\int_{\R^{N}}u^{p}\varphi_{R}dx\leq C\left(N,M,p\right)\left[R^{-s}\left(\int_{\R^{N}}u^{p}dx\right)^{\frac{1}{p}}+\left(\int_{\sqrt{R}\leq\left|x\right|\leq 2R}u^{p}dx\right)^{\frac{1}{p}}\right].
\end{align*}
In view of \eqref{5-1-6}, we again conclude that $u\equiv0$ by letting $R\rightarrow+\infty$.
\end{proof}

\begin{rem} \label{N13}
If $N=1$, Theorem \ref{5-1} reads as follows: if $s\geq\frac12$, then $u\equiv 0$ is the only nonnegative solution of \eqref{eq12} for any $p\geq1$; if instead $s<\frac12$, then the same conclusion holds provided $1\leq p\leq\frac{1}{1-2s}$.
\end{rem}

\section*{Acknowledgements} 
\noindent I.B. and G.G. are partially supported by  INdAM-GNAMPA. The authors wish to thank Sapienza University of Rome for the financial support from the research project Ateneo 2022 \lq\lq At the edge of reaction-diffusion equations: from population dynamics to geometric analysis, a qualitative approach\rq\rq.

\end{document}